\documentclass[12pt,a4paper]{amsart}
\usepackage{amssymb}
\usepackage[all,cmtip]{xy}

\calclayout

\newcommand{\+}{\nobreakdash-}
\renewcommand{\:}{\colon}

\newcommand{\rarrow}{\longrightarrow}
\newcommand{\larrow}{\longleftarrow}
\newcommand{\ot}{\otimes}

\DeclareMathOperator{\Hom}{Hom}
\DeclareMathOperator{\Ext}{Ext}
\DeclareMathOperator{\Tor}{Tor}
\DeclareMathOperator{\coker}{coker}
\DeclareMathOperator{\Ten}{Ten}
\DeclareMathOperator{\Sym}{Sym}

\newcommand{\Modl}{{\operatorname{\mathsf{--Mod}}}}
\newcommand{\Modr}{{\operatorname{\mathsf{Mod--}}}}
\newcommand{\Contra}{{\operatorname{\mathsf{--Contra}}}}
\newcommand{\Contrar}{{\operatorname{\mathsf{Contra--}}}}
\newcommand{\Comodl}{{\operatorname{\mathsf{--Comod}}}}
\newcommand{\Comodr}{{\operatorname{\mathsf{Comod--}}}}
\newcommand{\comodl}{{\operatorname{\mathsf{--comod}}}}
\newcommand{\comodr}{{\operatorname{\mathsf{comod--}}}}
\newcommand{\Vect}{{\operatorname{\mathsf{--Vect}}}}

\newcommand{\tors}{{\operatorname{\mathsf{-tors}}}}
\newcommand{\ctra}{{\operatorname{\mathsf{-ctra}}}}

\newcommand{\inj}{{\mathsf{inj}}}
\newcommand{\proj}{{\mathsf{proj}}}
\newcommand{\bb}{{\mathsf{b}}}
\newcommand{\co}{{\mathsf{co}}}
\newcommand{\ctr}{{\mathsf{ctr}}}

\newcommand{\rop}{{\mathrm{op}}}
\newcommand{\sop}{{\mathsf{op}}}

\newcommand{\lrarrow}{\mskip.5\thinmuskip\relbar\joinrel\relbar
   \joinrel\rightarrow\mskip.5\thinmuskip\relax}
\newcommand{\llarrow}{\mskip.5\thinmuskip\leftarrow\joinrel\relbar
   \joinrel\relbar\mskip.5\thinmuskip}
\newcommand{\bu}{{\text{\smaller\smaller$\scriptstyle\bullet$}}}

\newcommand{\sA}{\mathsf A}
\newcommand{\sB}{\mathsf B}
\newcommand{\sD}{\mathsf D}
\newcommand{\sE}{\mathsf E}

\newcommand{\boL}{\mathbb L}

\newcommand{\Section}[1]{\bigskip\section{#1}\medskip}
\theoremstyle{plain}
\newtheorem{thm}{Theorem}[section]
\newtheorem{prop}[thm]{Proposition}
\newtheorem{lem}[thm]{Lemma}

\newtheorem{cor}[thm]{Corollary}

\theoremstyle{definition}
\newtheorem{rem}[thm]{Remark}
\newtheorem{ex}[thm]{Example}

\begin{document}

\title{Homological full-and-faithfulness \\ of comodule inclusion \\
and contramodule forgetful functors}

\author{Leonid Positselski}

\address{Institute of Mathematics, Czech Academy of Sciences \\
\v Zitn\'a~25, 115~67 Prague~1 \\ Czech Republic} 

\email{positselski@math.cas.cz}

\begin{abstract}
 In this paper we consider a conilpotent coalgebra $C$ over a field~$k$.
 Let $\Upsilon\:C\Comodl\rarrow C^*\Modl$ be the natural functor of
inclusion of the category of $C$\+comodules into the category of
$C^*$\+modules, and let $\Theta\:C\Contra\rarrow C^*\Modl$ be
the natural forgetful functor.
 We prove that the functor $\Upsilon$ induces a fully faithful
triangulated functor on bounded (below) derived categories if and only
if the functor $\Theta$ induces a fully faithful triangulated functor on
bounded (above) derived categories, and if and only if the $k$\+vector
space $\Ext_C^n(k,k)$ is finite-dimensional for all $n\ge0$.
 We call such coalgebras ``weakly finitely Koszul''.
\end{abstract}

\maketitle

\tableofcontents

\section{Introduction}
\medskip

 In this paper we work with coassociative, counital coalgebras over
a field~$k$.
 For any such coalgebra $C$, the dual $k$\+vector space $C^*$ is 
naturally an associative, unital algebra over~$k$.
 One has to choose between two opposite ways of defining
the multiplication on~$C^*$.
 We prefer the notation in which any left $C$\+comodule becomes
a left $C^*$\+module, and any right $C$\+comodule becomes
a right $C^*$\+module.

 It is known, at least, since 1960s that the resulting exact functor
$\Upsilon\:C\Comodl\rarrow C^*\Modl$ is fully
faithful~\cite[Section~2.1]{Swe}.
 Following the terminology of the book~\cite{Swe}, $C^*$\+modules coming
from $C$\+comodules are often called ``rational'' in the literature.
 The essential image of the functor $\Upsilon$ is a \emph{hereditary
pretorsion class} in $C^*\Modl$: this means that the full subcategory
$\Upsilon(C\Comodl)$ is closed under subobjects, quotient objects,
and infinite coproducts in $C^*\Modl$.
 However, this full subcategory \emph{need not} be closed under
extensions.
 In other words, $\Upsilon(C\Comodl)$ is \emph{not} always a torsion
class or a Serre subcategory in $C^*\Modl$.
 A module extension of two comodules need not be a comodule.

 When is the essential image of $\Upsilon$ closed under extensions in
$C^*\Modl$\,?
 There is a vast body of literature on this topic, including
the papers~\cite{Rad,Shu,Lin,CNO,Cua,TT,Iov}.
 In this paper we discuss further questions going in this direction,
under an additional assumption.
 The assumption is that the coalgebra $C$ is \emph{conilpotent}.
 The conilpotent coalgebras were called ``pointed irreducible'' in
the terminology of~\cite{Swe}.
 For a conilpotent coalgebra $C$, the full subcategory $C\Comodl$
is closed under extensions in $C^*\Modl$ if and only if the coalgebra
$C$ is \emph{finitely cogenerated}~\cite[Corollary~2.4 and
Section~2.5]{Rad}, \cite[Theorem~4.6]{Shu}, \cite[Corollary~21]{Lin},
\cite[Theorem~2.8]{CNO}, \cite[Proposition~3.13 and
Corollary~3.14]{Cua}, \cite[Lemma~1.2 and Theorem~4.8]{Iov}.
 If $C$ is not finitely cogenerated, then there is a two-dimensional
$C^*$\+module which is not a $C$\+comodule, but an extension of two
one-dimensional $C$\+comodules.

 Let $\sA$ and $\sB$ be two abelian categories, and $\Phi\:\sB\rarrow
\sA$ be a fully faithful exact functor.
 Then the essential image of $\Phi$ is a full subcategory closed under
kernels and cokernels in~$\sA$.
 The full subcategory $\Phi(\sB)$ is closed under extensions in $\sA$ if
and only if the functor $\Phi$ induces an isomorphism
\begin{equation} \label{induced-Ext1-map}
 \Phi\:\Ext^1_\sB(X,Y)\lrarrow\Ext^1_\sA(\Phi(X),\Phi(Y))
\end{equation}
for all objects $X$, $Y\in\sB$.
 Generally speaking, for a fully faithful exact functor $\Phi$,
the map~\eqref{induced-Ext1-map} is a monomorphism, but not necessarily
an isomorphism.

 Thus, the following question is a natural extension of the question
about extension closedness of $\Upsilon(C\Comodl)$ in $C^*\Modl$.
 Put $\sB=C\Comodl$, \ $\sA=C^*\Modl$, and $\Phi=\Upsilon$.
 Consider the induced maps on the Ext spaces
\begin{equation} \label{induced-Ext-i-map}
 \Phi\:\Ext^i_\sB(X,Y)\lrarrow\Ext^i_\sA(\Phi(X),\Phi(Y)).
\end{equation}
 When is the map~\eqref{induced-Ext-i-map} in isomorphism for
all $X$, $Y\in\sB$\,?
 Generally speaking, for an exact functor of abelian categories
$\Phi\:\sB\rarrow\sA$, all one can say is that
the map~\eqref{induced-Ext-i-map} is a monomorphism for $i=n+1$ and
all $X$, $Y\in\sB$ whenever it is an isomorphism for $i=n$ and
all $X$, $Y\in\sB$.

 Let $C$ be a conilpotent coalgebra over a field~$k$.
 Then the one-dimensional $k$\+vector space $k$ has a unique left
$C$\+comodule structure (and a unique right $C$\+comodule structure)
provided by the unique coaugmentation of~$C$.
 In this context, we show that the maps~\eqref{induced-Ext-i-map} are
isomorphisms for $\Phi=\Upsilon$ and all $1\le i\le n$ if and only if
the $k$\+vector spaces $\Ext^i_C(k,k)$ (computed in the abelian
category of left or right $C$\+comodules) are finite-dimensional
for all $1\le i\le n$.
 In particular, $\Ext^1_C(k,k)$ is the vector space of cogenerators of
a conilpotent coalgebra~$C$; so $C$ is finitely cogenerated if and only
if $\Ext^1_C(k,k)$ is finite-dimensional.

 Alongside with the abelian categories of left and right
\emph{comodules} over a coalgebra $C$, there are much less familiar,
but no less natural abelian categories of left and right
\emph{$C$\+contramodules}~\cite{Prev}.
 Endowing the dual vector space $C^*$ to a coalgebra $C$ with
the natural algebra structure in which any left $C$\+comodule is a left
$C^*$\+module and any right $C$\+comodule is a right $C^*$\+module,
one  also obtains a natural left $C^*$\+module structure on any left
$C$\+contramodule.
 So there is an exact forgetful functor $\Theta\:C\Contra\rarrow
C^*\Modl$.

 The functor $\Theta$ is \emph{not} in general fully faithful.
 It was shown in the paper~\cite[Theorem~2.1]{Psm} that that
the functor $\Theta$ is fully faithful for any finitely cogenerated
conilpotent coalgebra~$C$.
 In this paper we demonstrate a counterexample proving the (much easier)
converse implication: if a conilpotent coalgebra $C$ is \emph{not}
finitely cogenerated, then the functor $\Theta$ is \emph{not} fully
faithful.

 More generally, the maps~\eqref{induced-Ext-i-map} are
isomorphisms for $\Phi=\Theta$ and all $0\le i\le n-1$ (for separated
contramodules $Y$; and also for $0\le i\le n-2$ and arbitrary~$Y$) if 
and only if the $k$\+vector spaces $\Ext^i_C(k,k)$ are
finite-dimensional for all $1\le i\le n$.
 Summarizing the assertions for comodules and contramodules, let us
state the following theorem.

\begin{thm} \label{cohomological-degree-careful-main-theorem}
 Let $C$ be a conilpotent coalgebra over a field~$k$ and $n\ge1$ be
an integer.
 Then the following five conditions are equivalent:
\begin{enumerate}
\renewcommand{\theenumi}{\roman{enumi}}
\item the map
$$
 \Ext^i_C(L,M)\lrarrow\Ext^i_{C^*}(L,M)
$$
induced by the inclusion functor\/ $\Upsilon\:C\Comodl\rarrow C^*\Modl$
is an isomorphism for all left $C$\+comodules $L$ and $M$, and
all\/ $1\le i\le n$;
\item the map
$$
 \Ext^i_{C^\rop}(L,M)\lrarrow\Ext^i_{C^*{}^\rop}(L,M)
$$
induced by the inclusion functor $\Comodr C\rarrow\Modr C^*$ is
an isomorphism for all right $C$\+comodules $L$ and $M$, and
all\/ $1\le i\le n$;
\item the map
\begin{equation} \label{contramodule-ext-map}
 \Ext^{C,i}(P,Q)\lrarrow\Ext^i_{C^*}(P,Q)
\end{equation}
induced by the forgetful functor\/ $\Theta\:C\Contra\rarrow C^*\Modl$
is an isomorphism for all left $C$\+contramodules $P$, all
\emph{separated} left $C$\+contramodules $Q$, and all\/ $0\le i\le n-1$;
\item the $k$\+vector space\/ $\Ext^i_C(k,k)$ is finite-dimensional
for all\/ $1\le i\le n$;
\item the $k$\+vector space\/ $\Ext^{C,i}(k,k)$ is finite-dimensional
for all\/ $1\le i\le n$.
\end{enumerate}
 If any one of the equivalent conditions~\textup{(i\+-v)} holds, then
the map~\eqref{contramodule-ext-map} is an isomorphism for \emph{all}
left $C$\+contramodules $P$ and $Q$ and all\/ $0\le i\le n-2$.
\end{thm}

 Notice a curious cohomological dimension shift in comparison between
the assertions about the comodule and contramodule inclusion/forgetful
functors.
 Let $n\ge1$ be the minimal integer for which the vector space
$\Ext^n_C(k,k)$ is infinite-dimensional.
 Then the map $\Ext^n_C(k,k)\rarrow\Ext^n_{C^*}(k,k)$
induced by the functor $\Upsilon$ is injective, but not surjective.
 In fact, the dimension cardinality of the vector space
$\Ext^n_{C^*}(k,k)$ is larger than that of $\Ext^n_C(k,k)$ in this case:
$\Ext^n_{C^*}(k,k)$ is as large as the double dual vector space
$\Ext^n_C(k,k)^{**}$ to $\Ext^n_C(k,k)$.

 At the same time, denoting by $T$ an infinite-dimensional
$k$\+vector space endowed with the trivial $C$\+contramodule structure,
the map $\Ext^{C,n}(T,k)\rarrow\Ext_{C^*}^n(T,k)$ induced by
the functor $\Theta$ is \emph{not injective} (for the integer~$n$
as in the previous paragraph).
 Consequently, there exists a projective $C$\+contramodule $P$ such
that the map $\Ext^{C,n-1}(P,k)\rarrow\Ext_{C^*}^{n-1}(P,k)$ is
injective, but not surjective.
 For $n>1$ this, of course, means that $\Ext_{C^*}^{n-1}(P,k)\ne0$,
while $\Ext^{C,n-1}(P,k)=0$ as it should be.
 Notice that both the contramodules $P$ and~$k$ are separated.

 Let us emphasize that \emph{we do not know} whether the equivalent
conditions~(i\+-v) of
Theorem~\ref{cohomological-degree-careful-main-theorem} imply
bijectivity of the maps~\eqref{contramodule-ext-map} for arbitrary
(nonseparated) contramodules $Q$ and $i=n-1$.
 This remains an open question.

 Returning to the discussion of a fully faithful exact functor
$\Phi\:\sB\rarrow\sA$, it is clear that the maps of Ext
groups~\eqref{induced-Ext-i-map} induced by $\Phi$ are isomorphisms if
and only if the induced triangulated functor between the bounded derived
categories $\Phi^\bb\:\sD^\bb(\sB)\rarrow\sD^\bb(\sA)$ is fully
faithful.
 Now assume that there are enough injective objects in the abelian
category $\sA$ and the functor $\Phi\:\sB\rarrow\sA$ has
a right adjoint.
 Then the functor $\Phi^\bb$ is fully faithful if and only if
the similar functor between the bounded below derived categories
$\Phi^+\:\sD^+(\sB)\rarrow\sD^+(\sA)$ is fully faithful.

 Dually, assume that there are enough projective objects in $\sA$ and
the functor $\Phi\:\sB\rarrow\sA$ has a left adjoint.
 Then the functor $\Phi^\bb$ is fully faithful if and only if
the similar functor between the bounded above derived categories
$\Phi^-\:\sD^-(\sB)\rarrow\sD^-(\sA)$ is fully
faithful~\cite[Proposition~6.5]{Pper}.

 In the situation at hand, there are enough projective and injective
objects in the categories of modules over associative rings.
 The comodule inclusion functor
$\Upsilon\:C\Comodl\rarrow C^*\Modl$ has a right adjoint functor
$\Gamma\:C^*\Modl\rarrow C\Comodl$, while the contramodule forgetful
functor $\Theta\:C\Contra\rarrow C^*\Modl$ has a left adjoint functor
$\Delta\:C^*\Modl\rarrow C\Contra$.
 Consequently,
Theorem~\ref{cohomological-degree-careful-main-theorem} implies
the following result about full-and-faithfulness of induced
triangulated functors.

\begin{thm} \label{bounded-half-unbounded-derived-main-theorem}
 For any conilpotent coalgebra $C$ over a field~$k$, the following
eight conditions are equivalent:
\begin{enumerate}
\renewcommand{\theenumi}{\roman{enumi}}
\item the triangulated functor\/
$\Upsilon^\bb\:\sD^\bb(C\Comodl)\rarrow\sD^\bb(C^*\Modl)$
induced by the comodule inclusion functor\/ $\Upsilon\:C\Comodl
\rarrow C^*\Modl$ is fully faithful;
\item the triangulated functor\/
$\Upsilon^+\:\sD^+(C\Comodl)\rarrow\sD^+(C^*\Modl)$
induced by the comodule inclusion functor\/ $\Upsilon$
is fully faithful;
\item the triangulated functor\/
$\sD^\bb(\Comodr C)\rarrow\sD^\bb(\Modr C^*)$
induced by the comodule inclusion functor $\Comodr C\rarrow\Modr C^*$
is fully faithful;
\item the triangulated functor\/
$\sD^+(\Comodr C)\rarrow\sD^+(\Modr C^*)$
induced by the comodule inclusion functor is fully faithful;
\item the triangulated functor\/
$\Theta^\bb\:\sD^\bb(C\Contra)\rarrow\sD^\bb(C^*\Modl)$
induced by the contramodule forgetful functor\/
$\Theta\:C\Contra\rarrow C^*\Modl$ is fully faithful;
\item the triangulated functor\/
$\Theta^-\:\sD^-(C\Contra)\rarrow\sD^-(C^*\Modl)$
induced by the contramodule forgetful functor\/ $\Theta$
is fully faithful;
\item the $k$\+vector space\/ $\Ext^n_C(k,k)$ is finite-dimensional
for all\/ $n\ge0$;
\item the $k$\+vector space\/ $\Ext^{C,n}(k,k)$ is finite-dimensional
for all\/ $n\ge0$.
\end{enumerate}
\end{thm}

 Let us discuss the condition of finite-dimensionality of the Ext
spaces $\Ext_C^i(k,k)$ for all $i\ge0$ in some detail.
 One can consider the special case when the coalgebra $C$ is positively
graded with finite-dimensional components; so $C$ is the graded dual
coalgebra to a positively graded algebra $A=\bigoplus_{m=0}^\infty A_m$
with $\dim A_m<\infty$ for all $m\ge0$ and $A_0=k$.
 Then one has $\Ext_C^i(k,k)\simeq\bigoplus_{j=i}^\infty
\Ext_A^{i,j}(k,k)$ and $\Ext_A^i(k,k)\simeq\prod_{j=i}^\infty
\Ext_A^{i,j}(k,k)$, where $i$~is the usual cohomological grading on
the Ext spaces, while the \emph{internal grading}~$j$ is induced by
the grading on~$A$ (cf.~\cite[Section~1 of Chapter~1]{PP},
\cite[Section~2.1]{Prel}, and~\cite[Section~2]{Pbogom}).

 Assume further that $A$ is multiplicatively generated by $A_1$ with
relations of degree~$2$; so $A$ is a quadratic algebra with
finite-dimensional components over~$k$.
 Then the vector spaces $\Ext_C^i(k,k)$ or $\Ext_A^i(k,k)$ still
need \emph{not} be finite-dimensional; in fact, $\Ext_C^3(k,k)$ or 
(equivalently) $\Ext_A^3(k,k)$ can be infinite-dimensional
already~\cite[Theorem~7.6]{An}, \cite{FGL},
\cite[Section~6 of Chapter~6]{PP}.

 Nevertheless, for a (\emph{homogeneous}) \emph{Koszul} $k$\+algebra $A$
with finite-dimensional grading components~\cite{Pr,PP} and the graded
dual coalgebra $C$, the vector spaces $\Ext_C^i(k,k)$ and
$\Ext_A^i(k,k)$ are, of course, finite-dimensional for all $i\ge0$.
 More generally, for a finitely cogenerated conilpotent coalgebra $C$,
all the vector spaces $\Ext_C^i(k,k)$ are finite-dimensional whenever
the Ext-algebra $\Ext_C^*(k,k)$ is generated by $\Ext_C^1(k,k)$.
 This includes the important particular case when the Ext-algebra
$\Ext_C^*(k,k)$ is Koszul (but the coalgebra $C$ need not be graded)
described in~\cite[Main Theorem]{PV} and~\cite[Sections~5\+-6]{Pqf}.
 Here we keep assuming that the coalgebra $C$ is finitely cogenerated.

 With these important special cases in mind, we call a conilpotent
coalgebra $C$ \emph{weakly finitely Koszul} if the vector space
$\Ext_C^n(k,k)$ is finite-dimensional for every $n\ge0$.
 This terminology goes back to~\cite[Sections~5.4 and~5.7]{Prel}.

 Let us emphasize that \emph{we do not know} whether the induced
triangulated functors between the unbounded derived categories
$$
 \Upsilon^\varnothing\:\sD(C\Comodl)\lrarrow\sD(C^*\Modl)
 \quad\text{and}\quad
 \Theta^\varnothing\:\sD(C\Contra)\lrarrow\sD(C^*\Modl)
$$
are fully faithful under the equivalent conditions of
Theorem~\ref{bounded-half-unbounded-derived-main-theorem}
(or even for a graded coalgebra $C$ graded dual to a Koszul algebra
$A$ with finite-dimensional components) in general.
 However, the triangulated functors $\Upsilon^\varnothing$ and
$\Theta^\varnothing$ are known to be fully faithful for any finitely
cogenerated conilpotent \emph{cocommutative} coalgebra~$C$.
 This is a particular case of~\cite[Theorems~1.3 and~2.9]{Pmgm}
(applied to the complete Noetherian commutative local $k$\+algebra
$R=C^*$ with its maximal ideal $I=\mathfrak m$).

 In this connection, it should first of all be mentioned that all
finitely cogenerated conilpotent cocommutative coalgebras $C$ are
weakly finitely Koszul since they are Artinian, and consequently,
co-Noetherian~\cite[Section~2]{Pmc}.
 Quite generally, any left or right co-Noetherian conilpotent coalgebra
is weakly finitely Koszul (and so is any finitely cogenerated left or
right cocoherent coalgebra).
 But this does not seem to be enough.
 The proofs of~\cite[Theorems~1.3 and~2.9]{Pmgm} are essentially based
on the observation that \emph{the right adjoint functor\/ $\Gamma$ to\/
$\Upsilon$ and the left adjoint functor\/ $\Delta$ to\/ $\Theta$ have
finite homological dimensions} (cf.~\cite[Theorem~6.4]{PMat}
and~\cite[Proposition~6.5]{Pper}).
 We cannot think of any general noncommutative versions of these
properties, proved in~\cite{PSY} and~\cite{Pmgm} using commutative
Koszul complexes.

\subsection*{Acknowledgement}
 I am grateful to Victor Roca i~Lucio and Teresa Conde for stimulating
correspondence.
 I~also wish to thank Jan \v St\!'ov\'\i\v cek for helpful
conversations.
 An acknowledgement is due to an anonymous referee for the suggestion
to include Remark~\ref{second-kind-not-relevant-remark}.
 The author is supported by the GA\v CR project 23-05148S and
the Czech Academy of Sciences (RVO~67985840).

\Section{Preliminaries on Coalgebras, Comodules, and Contramodules}
\label{preliminaries-on-co-and-contra-secn}

 Unless otherwise mentioned, all \emph{coalgebras}, \emph{comodules},
and \emph{contramodules} in this paper are presumed to be coassociative
and counital; all \emph{coalgebra homomorphisms} are presumed to
preserve the counit.
 Dually, all \emph{rings}, \emph{algebras}, and \emph{modules} are
presumed to be associative and unital.

 A (coassociative, counital) \emph{coalgebra} $C$ over a field~$k$ is
a $k$\+vector space endowed with $k$\+linear maps of
\emph{comultiplication} $\mu\:C\rarrow C\ot_k C$ and \emph{counit}
$\epsilon\:C\rarrow k$ satisfying the usual coassociativity and
counitality axioms (which can be obtained by writing down the definition
of an associative, unital $k$\+algebra in the tensor notation and
inverting the arrows).
 We suggest the books~\cite{Swe,Mon} as the standard reference sources
on coalgebras over a field.
 The present author's surveys~\cite[Section~1]{Prev},
\cite[Section~3]{Pksurv} can be used as additional reference sources.

 Let $C$ be a coalgebra over~$k$.
 A \emph{right $C$\+comodule} $N$ is a $k$\+vector space endowed with
a $k$\+linear map of \emph{right coaction} $\nu\:N\rarrow N\ot_k C$
satisfying the usual coassociativity and counitality axioms (which can
be obtained by inverting the arrows in the definition of a module over
an associative, unital algebra).
 Similarly, a \emph{left $C$\+comodule} $M$ is a $k$\+vector space
endowed with a $k$\+linear map of \emph{left coaction} $\nu\:M\rarrow
C\ot_k M$ satisfying the coassociativity and counitality axioms.

 A \emph{left $C$\+contramodule} (see~\cite[Section~III.5]{EM},
\cite[Sections~0.2.4 and~3.1.1, and Appendix~A]{Psemi},
\cite[Sections~1.1\+-1.6]{Prev}, and~\cite[Section~8]{Pksurv}) is
a $k$\+vector space $P$ endowed with a $k$\+linear map of \emph{left
contraaction} $\pi\:\Hom_k(C,P)\rarrow P$ satisfying the following
\emph{contraassociativity} and \emph{contraunitality} axioms.
 Firstly, the two maps $\Hom(C,\pi)$ and $\Hom(\mu,P)\:
\Hom_k(C\ot_kC,\>P)\simeq\Hom_k(C,\Hom_k(C,P))\rarrow\Hom_k(C,P)$
must have equal compositions with the contraaction map $\pi\:
\Hom_k(C,P)\rarrow P$,
$$
 \Hom_k(C\ot_kC,\>P)\simeq\Hom_k(C,\Hom_k(C,P))
 \,\rightrightarrows\,\Hom_k(C,P)\rarrow P.
$$
 Secondly, the composition of the map $\Hom(\epsilon,P)\:P\rarrow
\Hom_k(C,P)$ with the map $\pi\:\Hom_k(C,P)\rarrow P$ must be equal
to the identity endomorphism of~$P$,
$$
 P\rarrow \Hom_k(C,P)\rarrow P.
$$
 Here the $k$\+vector space isomorphism $\Hom_k(C\ot_kC,\>P)
\simeq\Hom_k(C,\Hom_k(C,P))$ is obtained as a particular case of
the adjunction isomorphism $\Hom_k(U\ot_kV,\>W)\simeq
\Hom_k(V,\Hom_k(U,W))$, which holds for any vector spaces $U$, $V$,
and~$W$.

 The definition of a \emph{right $C$\+contramodule} is similar, with
the only difference that the isomorphism $\Hom_k(C\ot_kC,\>P)
\simeq\Hom_k(C,\Hom_k(C,P))$ arising as a particular case of
the identification $\Hom_k(V\ot_kU,\>W)\simeq\Hom_k(V,\Hom_k(U,W))$
is used.

 For any right $C$\+comodule $N$ and any $k$\+vector space $V$,
the vector space $\Hom_k(N,V)$ has a natural left $C$\+contramodule
structure.
 The left contraaction map
$$
 \pi\:\Hom_k(C,\Hom_k(N,V))\simeq\Hom_k(N\ot_kC,\>V)
 \lrarrow\Hom_k(N,V)
$$
is constructed by applying the functor $\Hom_k({-},V)$ to the right
coaction map $\nu\:N\rarrow N\ot_kC$, i.~e., $\pi=\Hom_k(\nu,V)$.

 Let $C^\rop$ denote the opposite coalgebra to~$C$ (i.~e., the same
vector space endowed with the same counit and the left-right opposite
comultiplication to that in~$C$).
 Similarly, given a ring $A$, we denote by $A^\rop$ the opposite ring.
 Then right $C$\+comodules are the same things as left
$C^\rop$\+comodules, and vice versa.

 We will use the notation $\Hom_C({-},{-})$ for the vector spaces of
morphisms in the category of left $C$\+comodules $C\Comodl$ and
the notation $\Hom_{C^\rop}({-},{-})$ for the vector spaces of
morphisms in the category of right $C$\+comodules $\Comodr C$.
 The vector spaces of morphisms in the left and right contramodule
categories $C\Contra$ and $\Contrar C$ will be denoted by
$\Hom^C({-},{-})$ and $\Hom^{C^\rop}({-},{-})$.
 Similarly, the Yoneda Ext spaces are denoted by
$\Ext_C^*({-},{-})$ in $C\Comodl$, by $\Ext_{C^\rop}^*({-},{-})$
in $\Comodr C$, and by $\Ext^{C,*}({-},{-})$ in $C\Contra$.

 The category of left $C$\+comodules $C\Comodl$ is a locally finite
Grothendieck abelian category.
 Any $C$\+comodule is the union of its finite-dimensional
subcomodules~\cite[Propositions~2.1.1\+-2.1.2 and Corollary~2.1.4]{Swe},
\cite[Lemma~3.1(b)]{Pksurv}.
 The forgetful functor $C\Comodl\rarrow k\Vect$ from the category of
$C$\+comodules to the category of $k$\+vector spaces is exact and
preserves the infinite coproducts (but \emph{not}
the infinite products).
 So the coproduct (and more generally, filtered direct limit) functors
are exact in $C\Comodl$, while the functors of infinite product are
usually \emph{not} exact.

 Left $C$\+comodules of the form $C\ot_k V$ and right $C$\+comodules
of the form $V\ot_k C$, where $V$ ranges over the $k$\+vector spaces,
are called the \emph{cofree} $C$\+comodules.
 For any left $C$\+comodule $L$, the vector space of $C$\+comodule
morphisms $L\rarrow C\ot_k V$ is naturally isomorphic to the vector
space of $k$\+linear maps $L\rarrow V$,
$$
 \Hom_C(L,\>C\ot_kV)\simeq\Hom_k(L,V).
$$
 Hence the cofree comodules are injective (as objects of
$C\Comodl$ or $\Comodr C$).
 A $C$\+comodule is injective if and only if it is a direct summand of
a cofree one.

 The category of left $C$\+contramodules $C\Contra$ is a locally
presentable abelian category with enough projective objects.
 The forgetful functor $C\Contra\rarrow k\Vect$ is exact and preserves
the infinite products (but \emph{not} the infinite coproducts).
 Hence the functors of infinite product are exact in $C\Contra$, while
the coproduct functors are usually \emph{not} exact.

 Left $C$\+contramodules of the form $\Hom_k(C,V)$, where $V\in k\Vect$,
are called the \emph{free} $C$\+contramodules.
 For any left $C$\+contramodule $Q$, the vector space of
$C$\+contramodule morphisms $\Hom_k(C,V)\rarrow Q$ is naturally
isomorphic to the vector space of $k$\+linear maps $V\rarrow Q$,
$$
 \Hom^C(\Hom_k(C,V),Q)\simeq\Hom_k(V,Q).
$$
 Hence the free contramodules are projective (as objects of
$C\Contra$).
 A $C$\+con\-tra\-mod\-ule is projective if and only if it is a direct 
summand of a free one.

 Let us introduce a simplified version of the \emph{Sweedler
notation}~\cite[Section~1.2]{Swe}, \cite[Notation~1.4.2]{Mon} for
the comultiplication in~$C$.
 Given an element $c\in C$, we write
$$
 \mu(c)=c_{(1)}\ot c_{(2)}\in C\ot_k C.
$$
 Following the convention in~\cite{Psemi,Prev,Pksurv} (which is
opposite to the convention in~\cite{Swe,Mon}), we define
the associative algebra structure on the dual vector space
$C^*=\Hom_k(C,k)$ to a coalgebra $C$ by the formula
$$
 (fg)(c)=f(c_{(2)})g(c_{(1)})
 \qquad\text{for all $f$, $g\in C^*$ and $c\in C$}.
$$
 The counit on $C$ induces a unit in $C^*$ in the obvious way.

 Then, for any left $C$\+comodule $M$, the composition
$$
 C^*\ot_k M\lrarrow C^*\ot_k C\ot_k M\lrarrow M
$$
of the map induced by the coaction map $\nu\:M\rarrow C\ot_k M$
and the map induced by the pairing map $C^*\ot_k C\rarrow k$
endows $M$ with a left $C^*$\+module structure.
 Similarly, for any right $C$\+comodule $N$, the composition
$$
 N\ot_k C^*\lrarrow N\ot_k C\ot_k C^*\lrarrow N
$$
endows $N$ with a right $C^*$\+module structure.
 Finally, for any left $C$\+contramodule $P$, the composition
$$
 C^*\ot_k P\lrarrow\Hom_k(C,P)\lrarrow P
$$
of the natural embedding of vector spaces $C^*\ot_k P\rightarrowtail
\Hom_k(C,P)$ and the contraaction map $\pi\:\Hom_k(C,P)\rarrow P$
endows $P$ with a left $C^*$\+module structure.

 We have constructed the \emph{comodule inclusion functors}
\begin{align*}
 \Upsilon\:C\Comodl&\lrarrow C^*\Modl, \\
 \Comodr C&\lrarrow\Modr C^*
\end{align*}
and the \emph{contramodule forgetful functor}
$$
 \Theta\:C\Contra\lrarrow C^*\Modl.
$$
 The comodule inclusion functors (for a coalgebra $C$ over a field~$k$)
are always fully faithful (see~\cite[Propositions~2.1.1\+-2.1.2
and Theorem~2.1.3(e)]{Swe} for a discussion).
 The contramodule forgetful functor is \emph{not} fully faithful in
general, as we will see in Example~\ref{contramodule-extension-example}
below.

\Section{Conilpotent Coalgebras and Minimal Resolutions}
\label{conilpotent-minimal-secn}

 What we call \emph{conilpotent} coalgebras (in the terminology going
back to~\cite[Section~3.1]{PV}, \cite[Section~4.1]{Pbogom}) were called
``pointed irreducible'' coalgebras in~\cite[Section~8.0]{Swe}.
 We refer to~\cite[Sections~3.3\+-3.4]{Pksurv} for an introductory
discussion.

 Let $D$ be a coalgebra without counit over a field~$k$.
 For every $n\ge1$, there is the uniquely defined \emph{iterated
comultiplication map} $\mu^{(n)}\:D\rarrow D^{\ot n+1}$.
 The coalgebra $D$ is said to be \emph{conilpotent} if for every
$d\in D$ there exists $n\ge1$ such that $\mu^{(n)}(d)=0$
in $D^{\ot n+1}$.
 Clearly, one then also has $\mu^{(m)}(d)=0$ for all $m\ge n$.

 A \emph{coaugmentation}~$\gamma$ of a coalgebra $C$ is a homomorphism
of (counital) coalgebras $\gamma\:k\rarrow C$.
 So the composition $\epsilon\gamma\:k\rarrow C\rarrow k$ must be
the identity map.

 Given a coaugmented coalgebra $(C,\gamma)$, the cokernel
$D=C/\gamma(k)$ of the map~$\gamma$ has a unique coalgebra structure
for which the natural surjection $C\rarrow D$ is a homomorphism of
noncounital coalgebras.
 A coaugmented coalgebra $C$ us called \emph{conilpotent} if
the noncounital coalgebra $D$ is conilpotent (in the sense of
the definition above).

 Obviously, no nonzero noncounital coalgebra homomorphisms $k\rarrow D$
exist for a conilpotent noncounital coalgebra~$D$.
 Consequently, a conilpotent coaugmented coalgebra $(C,\gamma)$ admits
no other coaugmentation but~$\gamma$.

 The following result is a version of Nakayama lemma for conilpotent
noncounital coalgebras.

\begin{lem} \label{noncounital-nakayama}
\textup{(a)} Let $D$ be a conilpotent noncounital coalgebra and $M\ne0$
be a noncounital left $D$\+comodule.
 Then the coaction map $M\rarrow D\ot_kM$ is \emph{not}
injective. \par
\textup{(b)} Let $D$ be a conilpotent noncounital coalgebra and $P\ne0$
be a noncounital left $D$\+contramodule.
 Then the contraaction map\/ $\Hom_k(D,P)\rarrow P$ is \emph{not}
surjective.
\end{lem}

\begin{proof}
 Part~(a): for the sake of contradiction, assume that the coaction map
$\nu\:M\rarrow D\ot_kM$ is injective.
 Then the iterated coaction map $\nu^{(n)}\:M\rarrow D^{\ot n}\ot_kM$
is also injective for every $n\ge1$.

 Pick a nonzero element $x\in M$, and write $\nu(x)=
\sum_{i=1}^r d_r\ot y_r$ for some $d_i\in D$ and $y_i\in M$.
 Choose $n\ge1$ such that $\mu^{(n)}(d_i)=0$ in $D^{\ot n+1}$
for every $1\le i\le r$.
 Then $\nu^{(n+1)}(x)=\sum_{i=1}^r\mu^{(n)}(d_i)\ot y_i=0$ in
$D^{\ot n+1}\ot_k M$, a contradiction.

 The proof of part~(b) is a bit more involved; it can be found
in~\cite[Lemma~A.2.1]{Psemi}.
 For a discussion of generalizations and other versions of
the comodule and contramodule Nakayama lemmas,
see~\cite[Lemma~2.1]{Prev}.
\end{proof}

 Let $E$ be a subcoalgebra in a coalgebra~$C$.
 Then in any left $C$\+comodule $M$ there exists a unique maximal
subcomodule whose $C$\+comodule structure arises from
an $E$\+comodule structure.
 We denote this subcomodule, which can be computed as the kernel of
the composition of maps $M\rarrow C\ot_k M\rarrow C/E\ot_kM$,
by ${}_EM\subset M$.
 The similar subcomodule of a right $C$\+comodule $N$ will be
denoted by $N_E\subset N$.

 Dually, any left $C$\+contramodule $P$ admits a unique maximal
quotient contramodule whose $C$\+contramodule structure arises from
an $E$\+contramodule structure.
 We denote this quotient contramodule, which can be computed as
the cokernel of the composition of maps $\Hom_k(C/E,P)\rarrow
\Hom_k(C,P)\rarrow P$, by ${}^E\!P\twoheadleftarrow P$.
A further discussion of the functors $M\longmapsto{}_EM$ and
$P\longmapsto{}^E\!P$ can be found in~\cite[Section~2]{Pmc}
or~\cite[Section~8.4]{Pksurv}.

 For any right $C$\+comodule $N$, any subcoalgebra $E\subset C$,
and any $k$\+vector space $V$, there is a natural isomorphism of
left $E$\+contramodules
\begin{equation} \label{E-quotcontramodule-of-Hom-contramodule}
 {}^E\!\Hom_k(N,V)\simeq\Hom_k(N_E,V),
\end{equation}
where the left contramodule structure on the space of linear maps
from a right comodule to a vector space is constructed as explained
in Section~\ref{preliminaries-on-co-and-contra-secn}.
 The natural isomorphism~\eqref{E-quotcontramodule-of-Hom-contramodule}
follows immediately from the constructions of the functors
$M\longmapsto {}_EM$ and $P\longmapsto{}^E\!P$ above.

 In this section, we will be interested in the particular case when
$(C,\gamma)$ is a coaugmented (eventually, conilpotent) coalgebra
and $E=\gamma(k)\subset C$.
 In this case, we put ${}_\gamma M={}_EM$ and ${}^\gamma\!P={}^E\!P$.
 The similar notation for a right $C$\+comodule $N$ is
$N_\gamma=N_E$, and for a right $C$\+contramodule $Q$ it is
$Q^\gamma=Q^E$.
 Endowing the one-dimensional $k$\+vector space~$k$ with
the (left and right) $C$\+comodule and $C$\+contramodule structures
defined in terms of~$\gamma$, one has natural isomorphisms of
$k$\+vector spaces
$$
 {}_\gamma M\simeq\Hom_C(k,M) \quad\text{and}\quad
 \Hom_k({}^\gamma\!P,k)\simeq\Hom^C(P,k).
$$
 One can further compute the vector space ${}^\gamma\!P$ as
the \emph{contratensor product} $k\odot_CP$
(see~\cite[Section~3.1]{Prev} for the definition), but we will not
need to use this fact.

\begin{lem} \label{coaugmented-nakayama}
\textup{(a)} Let $(C,\gamma)$ be a conilpotent coaugmented coalgebra
and $M\ne0$ be a left $C$\+comodule.
 Then ${}_\gamma M\ne0$. \par
\textup{(b)} Let $(C,\gamma)$ be a conilpotent coaugmented coalgebra
and $P\ne0$ be a left $C$\+contramodule.
 Then ${}^\gamma\!P\ne0$.
\end{lem}

\begin{proof}
 This is an equivalent restatement of
Lemma~\ref{noncounital-nakayama} for $D=C/\gamma(k)$.
\end{proof}

 For a conilpotent coalgebra $C$, the subcomodule ${}_\gamma M$ of
a $C$\+comodule $M$ can be also described as the socle (i.~e.,
the maximal semisimple subcomodule) of $M$, and the quotient
contramodule ${}^\gamma\!P$ of a $C$\+contramodule $P$ can be
described as the cosocle (i.~e., the maximal semisimple quotient
contramodule) of~$P$.

\begin{lem} \label{mono-epi-characterized}
 Let $(C,\gamma)$ be a conilpotent coaugmented coalgebra.  Then \par
\textup{(a)} a morphism of left $C$\+comodules $f\:L\rarrow M$ is
injective if and only if the induced map of vector spaces
${}_\gamma f\:{}_\gamma L\rarrow{}_\gamma M$ is injective; \par
\textup{(b)} a morphism of left $C$\+contramodules $f\:P\rarrow Q$ is
surjective if and only if the induced map of vector spaces
${}^\gamma\!f\:{}^\gamma\!P\rarrow{}^\gamma Q$ is surjective.
\end{lem}

\begin{proof}
 Part~(a): if the map~$f$ is injective, then so is
the map~${}_\gamma f$, since ${}_\gamma N$ is a vector subspace in $N$
for every left $C$\+comodule~$N$.
 Conversely, the functor $N\longmapsto{}_\gamma N$ is left exact; so
if $K=\ker(f)\in C\Comodl$, then ${}_\gamma K=\ker({}_\gamma f)$.
 Now if ${}_\gamma K=0$, then $K=0$ by 
Lemma~\ref{coaugmented-nakayama}(a).
 The proof of part~(b) is similar (or rather, dual-analogous in
the sense of~\cite[Section~8.2]{Pksurv}).
 The functor $P\longmapsto{}^\gamma\!P$ is right exact, etc.
\end{proof}

\begin{lem} \label{inj-proj-isoms-characterized}
 Let $(C,\gamma)$ be a conilpotent coaugmented coalgebra.  Then \par
\textup{(a)} a morphism of injective left $C$\+comodules $f\:I\rarrow J$
is an isomorphism if and only if the induced map of vector spaces
${}_\gamma f\:{}_\gamma I\rarrow{}_\gamma J$ is an isomorphism; \par
\textup{(b)} a morphism of projective left $C$\+contramodules
$f\:P\rarrow Q$ is an isomorphism if and only if the induced map of
vector spaces ${}^\gamma\!f\:{}^\gamma\!P\rarrow{}^\gamma Q$ is
an isomorphism.
\end{lem}

\begin{proof}
 Part~(a): if the map~${}_\gamma f$ is injective, then the morphism~$f$
is injective by Lemma~\ref{mono-epi-characterized}(a).
 Now if the $C$\+comodule $I$ is injective, then $f$~is a split
monomorphism in $C\Comodl$.
 Therefore, ${}_\gamma\coker(f)=\coker({}_\gamma f)$, and
Lemma~\ref{coaugmented-nakayama}(a) tells us that $f$~is an isomorphism
whenever ${}_\gamma f$~is.
 The proof of part~(b) is dual-analogous.
\end{proof}

\begin{lem} \label{one-step-co-free-co-resolution}
 Let $(C,\gamma)$ be a conilpotent coaugmented coalgebra.  Then \par
\textup{(a)} for any left $C$\+comodule $M$ there exists a cofree
left $C$\+comodule $J$ together with an injective $C$\+comodule morphism
$f\:M\rarrow J$ such that the induced map ${}_\gamma f\:{}_\gamma M
\rarrow{}_\gamma J$ is an isomorphism of $k$\+vector spaces; \par
\textup{(b)} for any left $C$\+contramodule $Q$ there exists a free
left $C$\+contramodule $P$ together with a surjective $C$\+contramodule
morphism $f\:P\rarrow Q$ such that the induced map ${}^\gamma\!f\:
{}^\gamma\!P\rarrow{}^\gamma Q$ is an isomorphism of $k$\+vector spaces.
\end{lem}

\begin{proof}
 Part~(a): put $V={}_\gamma M$ and $J=C\ot_k V$.
 Then the injective map $V\rarrow J$ induced by the map $\gamma\:k
\rarrow C$ is a $C$\+comodule morphism $M\supset{}_\gamma M\rarrow J$.
 Since the cofree comodules are injective, this morphism can be
extended to a $C$\+comodule morphism $f\:M\rarrow J$.
 The morphism~$f$ is injective by Lemma~\ref{mono-epi-characterized}(a).
 The comodule $J$ can be also described as an injective envelope of
the comodule $M$ in the comodule category $C\Comodl$.
 The proof of part~(b) is dual-analogous; the contramodule $P$ can be
also described as a projective cover of the contramodule $Q$ in
contramodule category $C\Contra$ (cf.~\cite[Example~12.3]{Pproperf}).
\end{proof}

\begin{cor} \label{inj-proj-co-free-cor}
 Let $C$ be a conilpotent (coaugmented) coalgebra.  Then \par
\textup{(a)} a $C$\+comodule is injective if and only if is cofree; \par
\textup{(b)} a $C$\+contramodule is projective if and only if it is
free.
\end{cor}

\begin{proof}
 Generally speaking, over a coalgebra $C$ over a field~$k$,
the injective comodules are the direct summands of the cofree ones, and
the projective contramodules are the direct summands of the free ones
(see Section~\ref{preliminaries-on-co-and-contra-secn}).
 In the situation at hand with a conilpotent coalgebra $C$,
the assersion of part~(a) follows straighforwardly from
Lemmas~\ref{inj-proj-isoms-characterized}(a)
and~\ref{one-step-co-free-co-resolution}(a).
 Part~(b) follows from Lemmas~\ref{inj-proj-isoms-characterized}(b)
and~\ref{one-step-co-free-co-resolution}(b).
\end{proof}

 Let $K^\bu$ be a complex of left $C$\+comodules.
 We will say that the complex $K^\bu$ is \emph{minimal} if
the differential in the complex of vector spaces ${}_\gamma K^\bu$
vanishes.
 Similarly, a complex of left $C$\+contramodules $Q_\bu$ is said to be
\emph{minimal} if the differential in the complex of vector spaces
${}^\gamma Q_\bu$ vanishes.

 Let $M$ be a $C$\+comodule.
 An injective coresolution $M\rarrow J^\bu$ of the comodule $M$ is
said to be \emph{minimal} if the complex of comodules $J^\bu$
is minimal.
 Dually, a projective resolution $Q_\bu\rarrow P$ of a $C$\+contramodule
$P$ is said to be \emph{minimal} if the complex of contramodules
$Q_\bu$ is minimal.

\begin{prop} \label{minimal-co-resolutions-prop}
 The following assertions hold for any conilpotent coalgebra~$C$. \par
\textup{(a)} Any $C$\+comodule $M$ admits a minimal injective
coresolution.  A minimal injective coresolution of $M$ is unique up to
a (nonunique) isomorphism. \par
\textup{(b)} Any $C$\+contramodule $P$ admits a minimal projective
resolution.  A minimal projective resolution of $P$ is unique up to
a (nonunique) isomorphism.
\end{prop}

\begin{proof}
 Part~(a): to construct a minimal injective coresolution of $M$, all
one needs to do is to iterate the construction of
Lemma~\ref{one-step-co-free-co-resolution}(a).
 Pick an injective $C$\+comodule $J^0$ together with an injective
morphism of $C$\+comodules $M\rarrow J^0$ such that the induced map
${}_\gamma M\rarrow{}_\gamma J^0$ is an isomorphism.
 Then the induced map ${}_\gamma J^0\rarrow{}_\gamma(J^0/M)$ is zero.
 Pick an injective $C$\+comodule $J^1$ together with an injective
morphism of $C$\+comodules $J^0/M\rarrow J^1$ such that the induced
map ${}_\gamma(J^0/M)\rarrow{}_\gamma J^1$ is an isomorphism, etc.

 To prove uniqueness, let $M\rarrow I^\bu$ and $M\rarrow J^\bu$ be two
minimal injective coresolutions of~$M$.
 Then there exists a morphism of complexes of $C$\+comodules $f\:I^\bu
\rarrow J^\bu$ making the triangular diagram $M\rarrow I^\bu\rarrow
J^\bu$ commutative.
 Now we have ${}_\gamma I^\bu=\Hom_C(k,I^\bu)$ and ${}_\gamma J^\bu
=\Hom_C(k,I^\bu)$.
 The map of complexes $\Hom_C(k,f)\:\Hom_C(k,I^\bu)\rarrow
\Hom_C(k,J^\bu)$ induces an isomorphism of the cohomology spaces,
as both the complexes compute $\Ext_C^*(k,M)$.
 Since both the complexes ${}_\gamma I^\bu$ and ${}_\gamma J^\bu$ have
vanishing differentials by assumption, it follows that $\Hom_C(k,f)=
{}_\gamma f$ is a termwise isomorphism of complexes of vector spaces.
 Applying Lemma~\ref{inj-proj-isoms-characterized}(a), we conclude
that $f$~is a termwise isomorphism of complexes of $C$\+comodules.

 The proof of part~(b) is dual-analogous.
\end{proof}

\Section{The Left and Right Comodule and Contramodule Ext Comparison}

 Let $(C,\gamma)$ be a coaugmented coalgebra over a field~$k$
(as defined in Section~\ref{conilpotent-minimal-secn}).
 Then the one-dimensional $k$\+vector space $k$ can be endowed with
left and right $C$\+comodule and $C$\+contramodule structures
defined in terms of~$\gamma$.
 We recall the notation $\Ext_C^*({-},{-})$ for the Ext spaces
in the category of left $C$\+comodules $C\Comodl$ and
$\Ext_{C^\rop}^*({-},{-})$ for the Ext spaces in the category of
right $C$\+comodules $\Comodr C$, as well as the notation
$\Ext^{C,*}({-},{-})$ for the Ext spaces in the category of left
$C$\+contramodules $C\Contra$.

 The condition that the vector spaces $\Ext_C^n(k,k)$ be
finite-dimensional plays a key role in this paper.
 In this section we explain that this condition is left-right symmetric:
it holds for a coalgebra $C$ if and only if it holds for~$C^\rop$.
 Furthermore, the vector space $\Ext_C^n(k,k)$ is finite-dimensional
if and only if the vector space $\Ext^{C,n}(k,k)$ is.
 The following proposition is certainly essentially well-known; we spell
out the details here for the sake of completeness of the exposition.

\begin{prop} \label{ext-k-k-left-right-symmetric}
 Let $(C,\gamma)$ be a coaugmented coalgebra over~$k$.
 Then there are natural isomorphisms of Ext spaces
$$
 \Ext_C^n(k,k)\simeq\Ext_{C^\rop}^n(k,k)
 \qquad\text{for all\/ $n\ge0$}.
$$
 Moreover, in fact, the Ext-algebras\/ $\Ext_C^*(k,k)$ and\/
$\Ext_{C^\rop}^*(k,k)$ are naturally opposite to each other
(as graded algebras),
$$
 \Ext_{C^\rop}^*(k,k)\simeq\Ext_C^*(k,k)^\rop.
$$
\end{prop}

 The module analogue of Proposition~\ref{ext-k-k-left-right-symmetric}
is well-known: for any augmented associative algebra $A$ over
a field~$k$, there are natural isomorphisms of $k$\+vector spaces
$\Ext^n_A(k,k)\simeq\Tor^A_n(k,k)^*\simeq\Ext^n_{A^\rop}(k,k)$, making
$\Ext_{A^\rop}^*(k,k)$ the graded algebra with the opposite
multiplication to $\Ext_A^*(k,k)$.

\begin{proof}[First proof]
 Put $C_+=C/\gamma(k)$.
 For any left $C$\+comodule $M$, the reduced cobar coresolution
$$
 0\lrarrow M\lrarrow C\ot_k M\lrarrow C\ot_k C_+\ot_k M\lrarrow
 C\ot_k C_+\ot_k C_+\ot_k M\lrarrow\dotsb
$$
is an injective coresolution of $M$ in the abelian category $C\Comodl$.
 Applying the functor $\Hom_C(k,{-})$, we obtain a cobar complex
\begin{equation} \label{cobar-complex-of-comodule}
 M\lrarrow C_+\ot_k M\lrarrow C_+\ot_k C_+\ot_k M\lrarrow\dotsb
\end{equation}
computing the vector spaces $\Ext_C^*(k,M)$.
 In particular, the vector spaces $\Ext_C^*(k,k)$ are computed by
the cobar complex
\begin{equation} \label{cobar-complex-of-k}
 k\lrarrow C_+\lrarrow C_+\ot_k C_+\lrarrow C_+\ot_k C_+\ot_k C_+
 \lrarrow\dotsb
\end{equation}
\cite[Section~1.1]{PV}, \cite[Section~2.1]{Pbogom}.

 Reordering the tensor factors inside every term of the complex in
the opposite way identifies
the complex~\eqref{cobar-complex-of-k} for a coalgebra $C$ with
the similar complex for the opposite coalgebra~$C^\rop$.
 This suffices to prove the first assertion of the proposition.

 To prove the second assertion, one needs to observe that the obvious
(free associative) multiplication on the cobar
complex~\eqref{cobar-complex-of-k} induces a multiplication on its
cohomology spaces that is precisely opposite to the composition
multiplication on $\Ext_C^*(k,k)$.
 Moreover, the obvious left action of
the DG\+algebra~\eqref{cobar-complex-of-k} on
the complex~\eqref{cobar-complex-of-comodule} induces a left action
of the cohomology algebra of~\eqref{cobar-complex-of-k} on
the cohomology of~\eqref{cobar-complex-of-comodule} corresponding to
the natural graded right module structure over $\Ext_C^*(k,k)$
on $\Ext^*_C(k,M)$.
 Thus the isomorphism $\Ext_C^*(k,k)\simeq\Ext_{C^\rop}^*(k,k)$
obtained by comparing the cobar complexes~\eqref{cobar-complex-of-k}
for $C$ and $C^\rop$ makes the multiplication on the latter Ext-algebra
opposite to the one on the former one.

 This argument was sketched in~\cite[Section~2.4]{Pbogom}.
 An alternative proof of the first assertion of the proposition was
suggested in~\cite[Section~2.5]{Pbogom}.
\end{proof}

\begin{proof}[Second proof]
 Let $C\comodl$ and $\comodr C$ denote the abelian categories of
\emph{finite-dimensional} left and right $C$\+comodules, respectively.
 One makes two observations, which, taken together, imply both
the assertions of the proposition.

 Firstly, the fully faithful, exact inclusions of abelian categories
$C\comodl\rarrow C\Comodl$ and $\comodr C\rarrow\Comodr C$ induce
isomorphisms on the Ext spaces.
 Moreover, the triangulated functors between bounded above derived
categories $\sD^-(C\comodl)\rarrow\sD^-(C\Comodl)$ and
$\sD^-(\comodr C)\rarrow\sD^-(\Comodr C)$ induced by these inclusions
of abelian categories are fully faithful.
 This is a particular case of a well-known general result for abelian
or even exact categories (see~\cite[opposite version of
Proposition~1.7.11]{KS1}, \cite[opposite version of Theorem~13.2.8]{KS2},
\cite[opposite version of Theorem~12.1(b)]{Kel}, 
or~\cite[Proposition~2.1]{PSch}).
 The point is that for any finite-dimensional $C$\+comodule $L$,
any $C$\+comodule $M$, and any surjective $C$\+comodule morphism
$M\rarrow L$, there exists a finite-dimensional $C$\+subcomodule
$M'\subset M$ such that the composition $M'\rarrow M\rarrow L$
is surjective.

 Secondly, for any finite-dimensional right $C$\+comodule $N$, the dual
vector space $N^*$ is naturally a finite-dimensional left $C$\+comodule.
 Here it helps to notice that for any such $N$ there exists
a finite-dimensional subcoalgebra $E\subset C$ such that
the $C$\+comodule structure on $N$ arises from an $E$\+comodule
structure~\cite[Lemma~3.1(c)]{Pksurv}.
 Consequently, there is a natural anti-equivalence of abelian categories
$$
 N\longmapsto N^*\: (\comodr C)^\sop\lrarrow C\comodl
$$
taking the right $C$\+comodule~$k$ to the left $C$\+comodule~$k$.
 This anti-equivalence induces the desired anti-isomorphism of
the Ext algebras.
\end{proof}

\begin{lem} \label{co-contra-hom-double-dual-lemma}
 Let $C$ be a coalgebra over a field~$k$.
 Let $M\in\Comodr C$ be a right $C$\+comodule and $L\in\comodr C$ be
a finite-dimensional right $C$\+comodule.
 Then the vector space of left $C$\+contramodule homomorpisms
$M^*\rarrow L^*$ is naturally isomorphic to the double dual vector
space to the vector space of right $C$\+comodule homomorphisms
$L\rarrow M$,
\begin{equation} \label{comodule-contramodule-hom-double-dual}
 \Hom^C(M^*,L^*)\simeq\Hom_{C^\rop}(L,M)^{**}.
\end{equation}
 This isomorphism identifies the map\/ $\Hom_{C^\rop}(L,M)\rarrow
\Hom^C(M^*,L^*)$ induced by the contravariant functor of vector
space dualization $N\longmapsto N^*\:\Comodr C\rarrow C\Contra$
with the natural inclusion\/ $\Hom_{C^\rop}(L,M)\rarrow
\Hom_{C^\rop}(L,M)^{**}$ of a vector space\/ $\Hom_{C^\rop}(L,M)$
into its double dual vector space.
\end{lem}

\begin{proof}
 Notice that for any $k$\+vector space $V$ and any finite-dimensional
$k$\+vector space $W$ there is a natural isomorphism of $k$\+vector
spaces
\begin{equation} \label{double-dual-vector-space-hom-isomorphism}
 \Hom_k(V^*,W^*)\simeq\Hom_k(W,V)^{**},
\end{equation}
where $V^*=\Hom_k(V,k)$ and $W^*=\Hom_k(W,k)$.
 The isomorphism~\eqref{double-dual-vector-space-hom-isomorphism}
identifies the map $\Hom_k(W,V)\rarrow\Hom_k(V^*,W^*)$ induced
by the contravariant functor $U\longmapsto U^*\:k\Vect\rarrow k\Vect$
with the natural inclusion $\Hom_k(W,V)\rarrow\Hom_k(W,V)^{**}$
of a vector space $\Hom_k(W,V)$ into its double dual vector space.

 Now, for any two right $C$\+comodules $L$ and $M$, the vector space
$\Hom_{C^\rop}(L,M)$ is the kernel of (the difference of) a natural
pair of maps $\Hom_k(L,M)\rightrightarrows
\Hom_k(L,\>\allowbreak M\ot_k\nobreak C)$,
\begin{equation} \label{comodule-hom-computed}
 0\rarrow\Hom_{C^\rop}(L,M)\rarrow\Hom_k(L,M)\,\rightrightarrows
 \Hom_k(L,\>M\ot_kC).
\end{equation}
 For any two left $C$\+contramodules $P$ and $Q$, the vector space
$\Hom^C(P,Q)$ is the kernel of (the difference of) a natural pair of
maps $\Hom_k(P,Q)\rightrightarrows\Hom_k(\Hom_k(C,P),Q)$,
\begin{equation} \label{contramodule-hom-computed}
 0\rarrow\Hom^C(P,Q)\rarrow\Hom_k(P,Q)\,\rightrightarrows
 \Hom_k(\Hom_k(C,P),Q).
\end{equation}

 In particular, for $P=M^*$ and $Q=L^*$, we have
\begin{equation} \label{contramodule-hom-between-duals-computed}
 0\rarrow\Hom^C(M^*,L^*)\rarrow\Hom_k(M^*,L^*)\,\rightrightarrows
 \Hom_k(\Hom_k(C,M^*),L^*).
\end{equation}
 It remains to use the natural
isomorphism~\eqref{double-dual-vector-space-hom-isomorphism} in
order to show that the double dual vector space functor
$U\longmapsto U^{**}$ takes the rightmost pair of parallel morphisms
in the sequence~\eqref{comodule-hom-computed} to the rightmost pair of
parallel morphisms in
the sequence~\eqref{contramodule-hom-between-duals-computed}
for a finite-dimensional $C$\+comodule~$L$.
\end{proof}

\begin{prop} \label{ext-k-k-co-contra-prop}
 Let $(C,\gamma)$ be a coaugmented coalgebra over~$k$.
 Then the vector space\/ $\Ext^{C,n}(k,k)$ is naturally isomorphic to
the double dual vector space to the vector space\/
$\Ext_{C^\rop}^n(k,k)$,
$$
 \Ext^{C,n}(k,k)\simeq\Ext^n_{C^\rop}(k,k)^{**}
 \qquad\text{for all\/ $n\ge0$}.
$$
 This isomorphism identifies the map\/ $\Ext_{C^\rop}^n(k,k)\rarrow
\Ext^{C,n}(k,k)$ induced by the exact contravariant functor of vector
space dualization $N\longmapsto N^*\:\Comodr C\rarrow C\Contra$ with
the natural inclusion\/ $\Ext^n_{C^\rop}(k,k)\rarrow
\Ext^n_{C^\rop}(k,k)^{**}$ of a vector space\/ $\Ext^n_{C^\rop}(k,k)$
into its double dual vector space.
\end{prop}

\begin{proof}
 This observation goes back to the discussion
in~\cite[Section~A.1.2]{Psemi}.
 The point is that the functor $N\longmapsto N^*\:\Comodr C
\rarrow C\Contra$ takes injective right $C$\+comodules to projective
left $C$\+contramodules, which makes it easy to compute the induced
map of the Ext spaces.

 For any coalgebra $C$ and any right $C$\+comodule $M$, pick
an injective coresolution
\begin{equation} \label{right-comodule-injective-cores}
 0\lrarrow M\lrarrow J^0\lrarrow J^1\lrarrow J^2\lrarrow\dotsb
\end{equation}
of $M$ in the abelian category $\Comodr C$.
 Applying the dual vector space functor $\Hom_k({-},k)$
to the complex~\eqref{right-comodule-injective-cores}, we obtain
the complex
\begin{equation} \label{left-contra-proj-res-dual-to-comodule-cores}
 0\llarrow\Hom_k(M,k)\llarrow\Hom_k(J^0,k)\llarrow
 \Hom_k(J^1,k)\llarrow\dotsb,
\end{equation}
which is a projective resolution of the left $C$\+contramodule
$M^*=\Hom_k(M,k)$.

 Now let $L$ be a another right $C$\+comodule.
 Then one can compute the vector spaces $\Ext^n_{C^\rop}(L,M)$
as the cohomology spaces of the complex obtained by applying
the functor $\Hom_{C^\rop}(L,{-})$ to
the coresolution~\eqref{right-comodule-injective-cores}.
 Similarly, one can compute the vector spaces $\Ext^{C,n}(M^*,L^*)$
as the cohomology spaces of the complex obtained by applying
the functor $\Hom^C({-},L^*)$ to
the resolution~\eqref{left-contra-proj-res-dual-to-comodule-cores}.
 Taking into account Lemma~\ref{co-contra-hom-double-dual-lemma},
we obtain a natural isomorphism of complexes
$$
 \Hom^C((J^\bu)^*,L^*)\simeq\Hom_{C^\rop}(L,J^\bu)^{**}
$$
inducing a natural isomorphism of the cohomology spaces
$$
 \Ext^{C,n}(M^*,L^*)\simeq\Ext^n_{C^\rop}(L,M)^{**}
$$
for all right $C$\+comodules $M$, all finite-dimensional right
$C$\+comodules $L$, and all integers $n\ge0$.
 Moreover, it follows from the second assertion of the lemma that
the map $\Ext^n_{C^\rop}(L,M)\rarrow\Ext^{C,n}(M^*,L^*)$ induced by
the functor $N\longmapsto N^*\:\Comodr C\rarrow C\Contra$
agrees with the obvious map $\Ext^n_{C^\rop}(L,M)\rarrow
\Ext^n_{C^\rop}(L,M)^{**}$.
\end{proof}

\Section{Weakly Finite Koszulity Implies Comodule Ext Isomorphism}
\label{wf-koszulity-implies-comodule-ext-isom-secn}

 The aim of this section is to prove the following theorem.

\begin{thm} \label{wf-koszulity-implies-comodule-ext-isom-theorem}
 Let $C$ be a conilpotent coalgebra over a field~$k$ and $n\ge0$
be an integer such that the vector space $\Ext^i_C(k,k)$ is
finite-dimensional for all\/ $1\le i\le n$.
 Then the map of Ext spaces
$$
 \Ext_C^i(L,M)\lrarrow\Ext_{C^*}^i(L,M)
$$
induced by the inclusion functor\/ $\Upsilon\:C\Comodl\rarrow
C^*\Modl$ is an isomorphism for all left $C$\+comodules $L$ and $M$
and all\/ $0\le i\le n$.
\end{thm}

 We start with a lemma describing some vector spaces of morphisms in
the category of left $C^*$\+modules for a coalgebra~$C$.
 A right $C$\+comodule $N$ is said to be \emph{finitely cogenerated}
if it is a subcomodule of a cofree $C$\+comodule $V\ot_k C$ with
a finite-dimensional vector space of cogenerators~$V$.

\begin{lem} \label{C-star-hom-described}
 Let $C$ be a coalgebra over~$k$, let $N$ be a right $C$\+comodule,
and let $U$ be a $k$\+vector space.
 Then there is a natural monomorphism of $k$\+vector spaces
$$
 \eta_{N,U}\:N\ot_kU\lrarrow\Hom_{C^*}(N^*,\>C\ot_kU),
$$
which is an isomorphism whenever $N$ is a finitely cogenerated
injective $C$\+comodule.
 Here the left $C^*$\+module structure on $N^*=\Hom_k(N,k)$ is
obtained by dualizing the right $C^*$\+module structure on $N$
(or equivalently, by applying the forgetful functor to the natural
left $C$\+contramodule structure on~$N^*$), while the left
$C^*$\+module structure on $C\ot_kU$ comes from the cofree left
$C$\+comodule structure.
\end{lem}

\begin{proof}
 The map $\eta_{N,U}$ is defined by the formula
$$
 \eta_{N,U}(x\ot u)(f)=f(x_{(0)})x_{(1)}\ot u,
$$
for all $x\in N$, \ $u\in U$, and $f\in N^*$.
 Here $\nu(x)=x_{(0)}\ot x_{(1)}\in N\ot_k C$ is a notation for
the right $C$\+coaction map $\nu\:N\rarrow N\ot_kC$.

 To prove the isomorphism assertion for a finitely cogenerated
injective $C$\+comodule $N$, it suffices to consider the case
$N=V\ot_kC$, where $\dim_kV<\infty$.
 In this case, $N^*=C^*\ot_k V^*$ is a free left $C^*$\+module,
so $\Hom_{C^*}(N^*,\>C\ot_kU)=\Hom_k(V^*,\>C\ot_kU)=
V\ot_kC\ot_kU=N\ot_kU$.

 To prove that $\eta_{N,U}$~is an injective map for any $C$\+comodule
$N$, consider the $k$\+linear map
$$
 \Hom_{C^*}(N^*,\>C\ot_kU)\lrarrow\Hom_k(N^*,U)
$$
induced by the counit map $\epsilon\:C\rarrow k$.
 Notice that the counit is \emph{not} a $C$\+comodule morphism, and
consequently not a $C^*$\+module morphism, but only a linear map of
$k$\+vector spaces.
 The composition
$$
 N\ot_kU\overset\eta\lrarrow\Hom_{C^*}(N^*,\>C\ot_kU)
 \lrarrow\Hom_k(N^*,U)
$$
is the natural injective map of $k$\+vector spaces $N\ot_kU\rarrow
\Hom_k(N^*,U)$ (defined for any two vector spaces $U$ and $W=N$).
 Since the composition is a monomorphism, so is
the map~$\eta=\eta_{N,U}$.
\end{proof}

 Let $C=(C,\gamma)$ be a coaugmented coalgebra.
 For every $m\ge0$, denote by $F_mC\subset C$ the kernel of
the composition $C\rarrow C^{\ot m+1}\rarrow (C/\gamma(k))^{\ot m+1}$,
where $\mu^{(m)}\:C\rarrow C^{\ot m+1}$ is the iterated comultiplication
map and $C^{\ot m+1}\rarrow (C/\gamma(k))^{\ot m+1}$ is the natural
surjection.
 So one has $F_{-1}C=0$ and $F_0C=\gamma(k)\subset C$.
 One can check that $F$ is a comultiplicative filtration on $C$,
that is,
$$
 \mu(F_mC)\subset\sum\nolimits_{p+q=m} F_pC\ot_k F_qC\subset C\ot_kC
$$
for all $m\ge0$ \,\cite[Section~3.1]{PV}.

 Let $L$ be a left $C$\+comodule.
 Denote by $F_mL\subset L$ the full preimage of the subspace
$F_mC\ot_k L\subset C\ot_k L$ under the coaction map $\nu\:L\rarrow
C\ot_k L$.
 So one has $F_{-1}L=0$ and $F_0L={}_\gamma L$ (in the notation of
Section~\ref{conilpotent-minimal-secn}).
 One can check that $F$ is a comultiplicative filtration on $L$
compatible with the filtration $F$ on $C$, that is,
$$
 \nu(F_mL)\subset\sum\nolimits_{p+q=m} F_pC\ot_k F_qL\subset C\ot_k L
$$
for all $m\ge0$ \,\cite[Section~4.1]{Pbogom}.

 In other words, this means that $F_mL$ is a $C$\+subcomodule of $L$
for every $m\ge0$ \emph{and} the successive quotient $C$\+comodules
$F_mL/F_{m-1}L$ are trivial (i.~e., the coaction of $C$ in them is
induced by the coaugmentation~$\gamma$).
 In fact, one has $F_mL/F_{m-1}L={}_\gamma(L/F_{m-1}L)\subset
L/F_{m-1}L$ for every $m\ge1$.

 By the definition, a coaugmented coalgebra $C$ is conilpotent if and
only if $C=\bigcup_{m=0}^\infty F_mC$ (i.~e., the increasing filtration
$F$ on $C$ is exhaustive).
 In this case, it follows that $L=\bigcup_{m=0}^\infty F_mL$, i.~e.,
the increasing filtration $F$ on any $C$\+comodule is exhaustive
as well.

\begin{proof}[Proof of
Theorem~\ref{wf-koszulity-implies-comodule-ext-isom-theorem}]
 The argument resembles the proofs of the comodule Ext comparison
theorems in~\cite[Sections~5.4 and~5.7]{Prel} (see
specifically~\cite[Theorem~5.21]{Prel}).
 Notice first of all that the functor $\Upsilon$ is always fully
faithful by~\cite[Propositions~2.1.1\+-2.1.2 and Theorem~2.1.3(e)]{Swe}.

 According to Lemma~\ref{ext-preservation-enough-for-injectives} from
the appendix, in order to prove the theorem it suffices to show that
$\Ext^i_{C^*}(L,J)=0$ for all left $C$\+comodules $L$, all injective
left $C$\+comodules~$J$, and all integers $1\le i\le n$.
 Equivalently, this means that the space $\Ext^i_{C^*}(L,\>C\ot_kU)$
should vanish for all $k$\+vector spaces~$U$ and $1\le i\le n$.

 According to the discussion preceding this proof, the $C$\+comodule
$L$ has a natural increasing filtration with (direct sums of)
the trivial $C$\+comodule~$k$ as the successive quotients.
 Using the Eklof lemma~\cite[Lemma~1]{ET}, the problem reduces to
showing that $\Ext^i_{C^*}(k,\>C\ot_kU)=0$ for all $k$\+vector
spaces~$U$.

 Let
\begin{multline} \label{minimal-coresolution}
 0\lrarrow k\lrarrow V_0\ot_k C\lrarrow V_1\ot_k C\lrarrow V_2\ot_k C
 \lrarrow\dotsb \\ \lrarrow V_{n-1}\ot_k C\lrarrow V_n\ot_k C
 \overset{\tau_n}\lrarrow V_{n+1}\ot_k C\lrarrow\dotsb
\end{multline}
be a minimal injective/cofree coresolution of the right
$C$\+comodule~$k$, as per Corollary~\ref{inj-proj-co-free-cor}(a)
and Proposition~\ref{minimal-co-resolutions-prop}(a).
 Here $V_i$, \,$i\ge0$, are some $k$\+vector spaces.
 Computing the Ext spaces $\Ext_{C^\rop}^*(k,k)$ using the injective
coresolution~\eqref{minimal-coresolution}, we obtain isomorphisms
$V_i\simeq\Ext_{C^\rop}^i(k,k)$ (in particular, $V_0\simeq k$).
 Furthermore, Proposition~\ref{ext-k-k-left-right-symmetric}
provides isomorphisms $\Ext_{C^\rop}^i(k,k)\simeq\Ext_C^i(k,k)$.
 So the vector space $V_i$ is finite-dimensional for $0\le i\le n$
by assumption.

 Applying the dual vector space functor $\Hom_k({-},k)$ to
the complex of right $C$\+comodules~\eqref{minimal-coresolution},
we obtain a resolution of the left $C^*$\+module~$k$,
\begin{multline} \label{dual-initially-projective-resolution}
 0\llarrow k\llarrow C^*\ot_kV_0^*\llarrow C^*\ot_kV_1^*\llarrow
 C^*\ot_kV_2^* \llarrow\dotsb \\
 \llarrow C^*\ot_kV_{n-1}^*\llarrow C^*\ot_kV_n^*
 \llarrow (V_{n+1}\ot_k C)^*\llarrow\dotsb
\end{multline}
 The exact complex of left
$C^*$\+modules~\eqref{dual-initially-projective-resolution} starts and
proceeds up to the homological degree~$n$ as a resolution by finitely
generated projective $C^*$\+modules before turning into a resolution
by some arbitrary $C^*$\+modules (in fact, projective left
$C$\+contramodules) from the homological degree~$n+1$ on.

 For every $0\le i\le n+1$ and any left $C^*$\+module $J$, the space
$\Ext^i_{C^*}(k,J)$ can be computed as the degree~$i$ cohomology space
of the complex obtained by applying the functor $\Hom_{C^*}({-},J)$ to
the resolution~\eqref{dual-initially-projective-resolution} of
the left $C^*$\+module~$k$ (see
Lemma~\ref{initially-projective-resolution-computes-ext}).
 We are only interested in $0\le i\le n$ now, so let us write down
the resulting complex in the cohomological degrees from~$0$ to~$n+1$.
 It has the form
$$
 0\rarrow V_0\ot_k J\rarrow V_1\ot_k J\rarrow\dotsb
 \rarrow V_n\ot_k J\rarrow\Hom_{C^*}((V_{n+1}\ot_k C)^*,\>J),
$$
which for $J=C\ot_kU$ turns into
\begin{multline} \label{computing-ext-by-init-proj-for-cofree}
 0\lrarrow V_0\ot_k C\ot_k U\lrarrow V_1\ot_k C\ot_k U\lrarrow\dotsb
 \\ \lrarrow V_n\ot_k C\ot_k U\overset\theta\lrarrow
 \Hom_{C^*}((V_{n+1}\ot_k C)^*,\>C\ot_kU).
\end{multline}

 In the cohomological degrees~$\le n$,
the complex~\eqref{computing-ext-by-init-proj-for-cofree} can be
simply obtained by applying the vector space tensor product functor
${-}\ot_kU$ to the coresolution~\eqref{minimal-coresolution}.
 This follows from Lemma~\ref{C-star-hom-described}.
 Consequently, the complex~\eqref{computing-ext-by-init-proj-for-cofree}
is exact in the cohomological degrees $0<i<n$.
 It is only the complicated rightmost term
of~\eqref{computing-ext-by-init-proj-for-cofree} that remains
to be dealt with.

 Finally, the same Lemma~\ref{C-star-hom-described} provides
a commutative square diagram of $k$\+linear maps
\begin{equation} \label{bad-term-comparison-square}
\begin{gathered}
 \xymatrix{
  V_n\ot_k C\ot_k U \ar[r]^{\tau_n\ot_k U} \ar@{=}[d]_\eta
  \ar[rd]^\theta
  & V_{n+1}\ot_k C\ot_k U \ar@{>->}[d]^\eta \\
  \Hom_{C^*}((V_n\ot_k C)^*,\>C\ot_kU) \ar[r] &
  \Hom_{C^*}((V_{n+1}\ot_k C)^*,\>C\ot_kU)
 }
\end{gathered}
\end{equation}
where the leftmost vertical isomorphism is $\eta=\eta_{V_n\ot_kC,\>U}$,
the rightmost vertical monomorphism is $\eta=\eta_{V_{n+1}\ot_kC,\>U}$,
the horizontal arrows are induced by the differential~$\tau_n$
in the coresolution~\eqref{minimal-coresolution},
and the diagonal composition~$\theta$ is the rightmost differential
in the complex~\eqref{computing-ext-by-init-proj-for-cofree}.
 Now it is clear from the diagram~\eqref{bad-term-comparison-square}
that the kernel of the map $\theta\:V_n\ot_k C\ot_k U\rarrow
\Hom_{C^*}((V_{n+1}\ot_k C)^*,\>C\ot_kU)$ coincides with
the kernel of the map $\tau_n\ot_kU\:V_n\ot_k C\ot_k U\rarrow
V_{n+1}\ot_k C\ot_k U$ (because the rightmost vertical map~$\eta$ is
injective).
 Thus exactness of the coresolution~\eqref{minimal-coresolution} implies
exactness of the complex~\eqref{computing-ext-by-init-proj-for-cofree}
in the cohomological degree~$n$ as well.
\end{proof}

\begin{cor} \label{comodule-ext-injective-map-cor}
 Let $C$ be a conilpotent coalgebra over a field~$k$ and $n\ge0$ be
an integer such that the vector space\/ $\Ext_C^i(k,k)$ is
finite-dimensional for all\/ $1\le i\le n$.
 Then the map of Ext spaces
$$
 \Ext_C^i(L,M)\lrarrow\Ext_{C^*}^i(L,M)
$$
induced by the comodule inclusion functor\/ $\Upsilon\:C\Comodl
\rarrow C^*\Modl$ is injective for all left $C$\+comodules $L$ and $M$
and all\/ $0\le i\le n+1$.
\end{cor}

\begin{proof}
 This is a purely formal consequence of
Theorem~\ref{wf-koszulity-implies-comodule-ext-isom-theorem}.
 See Lemma~\ref{induced-on-Ext-1-lemma} in the appendix for the case
$n=0$, and Lemma~\ref{next-degree-ext-injectivity-lemma}(a) for
the general case.
\end{proof}

\Section{Comodule Ext Isomorphism Implies Weakly Finite Koszulity}
\label{comodule-ext-isom-implies-wf-koszulity-secn}

 The aim of this section is to prove the following theorem.

\begin{thm} \label{comodule-ext-isom-implies-wf-koszulity-theorem}
 Let $C$ be a conilpotent coalgebra over a field~$k$ and $n\ge1$
be the minimal integer for which the vector space $\Ext_C^n(k,k)$ is
infinite-dimensional.
 Then the map\/ $\Ext^n_C(k,k)\rarrow\Ext^n_{C^*}(k,k)$ induced by
the comodule inclusion functor\/ $\Upsilon\:C\Comodl\rarrow C^*\Modl$
is injective, but \emph{not} surjective.
 In fact, if $\lambda$~is the dimension cardinality of the $k$\+vector
space\/ $\Ext^n_C(k,k)$ and $\kappa$~is the cardinality of
the field~$k$, then the dimension cardinality of\/
$\Ext^n_{C^*}(k,k)$ is equal to~$\kappa^{\kappa^\lambda}$.
\end{thm}

 The following lemma computing the dimension cardinality of the dual
vector space is a classical result.

\begin{lem} \label{dimension-of-the-dual-lemma}
 Let $V$ be an infinite-dimensional vector space of dimension
cardinality~$\lambda$ over a field~$k$ of cardinality~$\kappa$.
 Then the dimension cardinality of the dual $k$\+vector space
$V^*=\Hom_k(V,k)$ is equal to $\kappa^\lambda$.
\end{lem}

\begin{proof}
 This is~\cite[Section~IX.5]{Jac}.
\end{proof}

 Let $C$ be a conilpotent coalgebra.
 Then the vector space $\Ext^1_C(k,k)$ can be interpreted as
the \emph{space of cogenerators} of the coalgebra~$C$
(we refer to~\cite[Lemma~5.2]{Pqf} for a discussion).
 So we will say that a conilpotent coalgebra $C$ is \emph{finitely
cogenerated} if the vector space $\Ext^1_C(k,k)$ is finite-dimensional.

 For any coaugmented coalgebra $(C,\gamma)$, the dual linear map
$\gamma^*\:C^*\rarrow k$ to the coaugmentation $\gamma\:k\rarrow C$
defines an augmentation on the algebra~$C^*$.
 Accordingly, the one-dimensional vector space~$k$ can be endowed
with left and right $C^*$\+module structures provided by
the augmentation~$\gamma^*$.
 These $C^*$\+module structures on~$k$ can be also viewed as coming
from the left and right $C$\+comodule structures on~$k$ induced
by the coaugmentation~$\gamma$.
 The same module structures also come from the left and right
$C$\+contramodule structures on~$k$ induced by~$\gamma$.

 We start with a discussion of the case $n=1$ in
Theorem~\ref{comodule-ext-isom-implies-wf-koszulity-theorem}.

\begin{prop} \label{comodule-ext-1-isom-implies-prop}
 Let $C$ be a conilpotent coalgebra over a field~$k$.
 Then there is a natural $k$\+vector space monomorphism
\begin{equation} \label{Ext-1-double-dual-monomorphism}
 \Ext_C^1(k,k)^{**}\,\rightarrowtail\,\Ext_{C^*}^1(k,k).
\end{equation}
 The dimension cardinality of the vector space\/ $\Ext_{C^*}^1(k,k)$
is equal to that of the vector space\/ $\Ext_C^1(k,k)^{**}$.
\end{prop}

\begin{proof}
 Notice first of all the isomorphism $\Ext_{C^\rop}^*(k,k)\simeq
\Ext_C^*(k,k)$ provided by
Proposition~\ref{ext-k-k-left-right-symmetric}.
 Furthermore, let $A$ be an augmented $k$\+algebra with
an augmentation $\alpha\:A\rarrow k$ and the augmentation ideal
$A_+=\ker(\alpha)\subset A$.
 Then the vector space $\Ext^1_A(k,k)\simeq\Tor_1^A(k,k)^*\simeq
\Ext^1_{A^\rop}(k,k)$ is computed as the kernel of the $k$\+linear map
$$
 A_+^*\lrarrow(A_+\ot_k A_+)^*
$$
dual to the multiplication map $A_+\ot_k A_+\rarrow A_+$.
 In particular, the vector space $\Ext^1_{C^*}(k,k)$ is
naturally isomorphic to the kernel of the map
\begin{equation} \label{ext-1-over-dual-algebra-computed}
 C_+^{**}\lrarrow(C_+^*\ot_k C_+^*)^*
\end{equation}
dual to the multiplication map $C_+^*\ot_k C_+^*\rarrow C_+^*$
(where $C_+=C/\gamma(k)$).

 On the other hand, from the cobar complex~\eqref{cobar-complex-of-k}
one can immediately see that the vector space $\Ext_C^1(k,k)$
is naturally isomorphic to the kernel of the comultiplication map
$$
 C_+\lrarrow C_+\ot_kC_+.
$$
 Hence the double dual vector space $\Ext_C^1(k,k)^{**}$ is
naturally isomorphic to the kernel of the map
\begin{equation} \label{double-dual-ext-1-over-coalgebra}
 C_+^{**}\lrarrow(C_+\ot_kC_+)^{**}.
\end{equation}

 Notice that $(C_+^*\ot_kC_+^*)$ is naturally a subspace
in $(C_+\ot_kC_+)^*$; hence $(C_+^*\ot_kC_+^*)^*$ is a quotient
space of $(C_+\ot_kC_+)^{**}$.
 One can also observe that the cokernel of the map $(C_+\ot_kC_+)^*
\rarrow C_+^*$ is naturally a quotient space of the cokernel of the map
$C_+^*\ot_kC_+^*\rarrow C_+^*$.
 Comparing the maps~\eqref{ext-1-over-dual-algebra-computed}
and~\eqref{double-dual-ext-1-over-coalgebra}, one immediately obtains
the desired monomorphism~\eqref{Ext-1-double-dual-monomorphism}.

 To compute the dimension cardinality, put $V=\Ext^1_C(k,k)$, and
notice that $C$ is a subcoalgebra of the cofree conilpotent (tensor)
coalgebra $\Ten(V)=\bigoplus_{n=0}^\infty V^{\ot n}$ cospanned by~$V$
\,\cite[Lemma~5.2]{Pqf} (see~\cite[Sections~2.3 and~3.3]{Pksurv} for
an introductory discussion).
 Hence, if $V$ is infinite-dimensional, then the dimension cardinalities
of $V$ and $C_+$ are equal to each other.
 Now $V^{**}$ is a subspace in $\Ext^1_{C^*}(k,k)$, which is in
turn a subspace in $C_+^{**}$.
 Thus the dimension cardinalities of all the three vector spaces
$V^{**}$, \ $\Ext^1_{C^*}(k,k)$, and $C_+^{**}$ are equal to
each other.
 If $V$ is finite-dimensional, then
the map~\eqref{Ext-1-double-dual-monomorphism} is an isomorphism
by Theorem~\ref{wf-koszulity-implies-comodule-ext-isom-theorem}
(for $n=1$).
\end{proof}

\begin{ex} \label{noncomodule-extension-example}
 Let $C$ be an infinitely cogenerated conilpotent coalgebra.
 Interpreted in the light of Lemma~\ref{induced-on-Ext-1-lemma} from
the appendix, Proposition~\ref{comodule-ext-1-isom-implies-prop} tells
us that there exists a short exact sequence of left $C^*$\+modules
\begin{equation} \label{two-dimensional-noncomodule}
 0\lrarrow k\lrarrow M\lrarrow k\lrarrow0
\end{equation}
with the two-dimensional $C^*$\+module $M$ \emph{not} coming from
any $C$\+comodule (while the one-dimensional $C^*$\+module $k$,
of course, comes from the one-dimensional $C$\+comodule~$k$) via
the comodule inclusion functor~$\Upsilon$.
 Let us explain how to construct a short exact
sequence~\eqref{two-dimensional-noncomodule} explicitly.

 Let $C$ be a conilpotent coalgebra.
 Recall the notation $F_0C=\gamma(k)$ and $F_1C=\ker(C\to C_+\ot_kC_+)$
from Section~\ref{wf-koszulity-implies-comodule-ext-isom-secn}.
 Then we have a natural direct sum decomposition $F_1C=F_0C\oplus V$,
where $V=\ker(C_+\to C_+\ot_kC_+)=\Ext^1_C(k,k)$.
 Choose a linear function $f\:V^*\rarrow k$.
 The composition of linear maps $C^*\twoheadrightarrow(F_1C)^*
\twoheadrightarrow V^*\overset f\rarrow k$ defines a linear
function $\tilde f\:C^*\rarrow k$.

 Let $M$ be the two-dimensional $k$\+vector space with the basis
vectors~$e_1$ and~$e_2$.
 Define the left action of $C^*$ in $M$ by the formulas
$ae_1=\gamma^*(a)e_1$ and $ae_2=\gamma^*(a)e_2+\tilde f(a)e_1$
for all $a\in C^*$ (where $\gamma^*\:C^*\rarrow k$ is the dual
map to the coaugmentation~$\gamma$).
 Then $ke_1\subset M$ is a $C^*$\+submodule of $M$ isomorphic to
$\Upsilon(k)$ and $M/ke_1$ is a quotient module of $M$ also
isomorphic to $\Upsilon(k)$.
 The $C^*$\+module $M$ itself belongs to the essential image
of the functor $\Upsilon$ if and only if the linear function
$f\:V^*\rarrow k$ comes from a vector in~$V$.
\end{ex}

 Before passing to the case $n>1$ in
Theorem~\ref{comodule-ext-isom-implies-wf-koszulity-theorem},
we need a preparatory lemma.

\begin{lem} \label{hom-from-dual-into-k}
 Let $(C,\gamma)$ be a coaugmented coalgebra over~$k$ and $N$ be
a right $C$\+comodule.
 Then there is a natural monomorphism of $k$\+vector spaces
$$
 \xi_N\:(N_\gamma)^{**}\lrarrow\Hom_{C^*}(N^*,k)
$$
from the double dual vector space to the vector subspace
$N_\gamma\subset N$ to the vector space of left $C^*$\+module
morphisms $N^*\rarrow k$.
 The map~$\xi_N$ is an isomorphism whenever the coalgebra $C$ is
conilpotent and finitely cogenerated. 
\end{lem}

\begin{proof}
 The inclusion $N_\gamma\rarrow N$ is a morphism of right
$C$\+comodules.
 Consequently, the $k$\+vector space dual map $N^*\rarrow (N_\gamma)^*$
is an epimorphism of left $C^*$\+modules (and in fact, of left
$C$\+contramodules), where the $C^*$\+module (or $C$\+contramodule)
structure on $(N_\gamma)^*$ is induced by the (co)augmentation.
 Applying the functor $\Hom_{C^*}({-},k)$, we obtain the desired
injective map of $k$\+vector spaces
$$
 \xi_N\:(N_\gamma)^{**}=\Hom_{C^*}((N_\gamma)^*,k)
 \lrarrow\Hom_{C^*}(N^*,k).
$$

 If the coalgebra $C$ is conilpotent and finitely cogenerated,
then the contramodule forgetful functor $C\Contra\rarrow C^*\Modl$
from the category of left $C$\+contramodules to the category of
left $C^*$\+modules is fully faithful~\cite[Theorem~2.1]{Psm}
(see Theorem~\ref{fully-faithful-contramodule-forgetful} below).
 Consequently, we have $\Hom_{C^*}(N^*,k)=\Hom^C(N^*,k)$.

 It remains to observe that the natural map
\begin{equation} \label{contra-hom-from-dual-into-k-eqn}
 (N_\gamma)^{**}\lrarrow\Hom^C(N^*,k)
\end{equation}
is an isomorphism for any coaugmented coalgebra~$(C,\gamma)$.
 Indeed, we have
$$
 \Hom^C(N^*,k)\simeq\Hom_k({}^\gamma(N^*),k)
 \simeq\Hom_k((N_\gamma)^*,k)=(N_\gamma)^{**},
$$
because ${}^\gamma(N^*)\simeq(N_\gamma)^*$.
 The latter isomorphism is a particular case of
the formula~\eqref{E-quotcontramodule-of-Hom-contramodule}
in Section~\ref{conilpotent-minimal-secn}.
 Alternatively, one can obtain
the isomorphism~\eqref{contra-hom-from-dual-into-k-eqn} as
a particular case of
the isomorphism~\eqref{comodule-contramodule-hom-double-dual}
from Lemma~\ref{co-contra-hom-double-dual-lemma} for $L=k$ and $M=N$.
\end{proof}

\begin{prop} \label{comodule-ext-n-isom-implies-prop}
 Let $C$ be a finitely cogenerated conilpotent coalgebra over
a field~$k$ and $n\ge2$ an integer such that the vector space\/
$\Ext^i_C(k,k)$ is finite-dimensional for all\/ $1\le i\le n-1$.
 Then there is a natural $k$\+vector space isomorphism
\begin{equation} \label{Ext-n-double-dual-isomorphism}
 \Ext_C^n(k,k)^{**}\simeq\Ext_{C^*}^n(k,k).
\end{equation}
\end{prop}

\begin{proof}
 Similarly to the proof of
Theorem~\ref{wf-koszulity-implies-comodule-ext-isom-theorem}, we
consider a minimal injective/cofree
coresolution~\eqref{minimal-coresolution} of
the right $C$\+comodule~$k$.
 The assumption of the proposition implies that the vector spaces
$V_i\simeq\Ext^i_C(k,k)$ are finite-dimensional for all $0\le i\le n-1$.
 Applying the dual vector space functor $\Hom_k({-},k)$ to
the complex~\eqref{minimal-coresolution}, we obtain a resolution
of the left $C^*$\+module~$k$,
\begin{multline} \label{dual-initially-less-projective-resolution}
 0\llarrow k\llarrow C^*\ot_kV_0^*\llarrow C^*\ot_kV_1^*\llarrow
 C^*\ot_kV_2^* \llarrow\dotsb \\
 \llarrow C^*\ot_kV_{n-1}^*\llarrow (V_n\ot_k C)^*
 \llarrow (V_{n+1}\ot_k C)^*\llarrow\dotsb
\end{multline}
(cf.\ formula~\eqref{dual-initially-projective-resolution},
which holds under the stricter assumptions of
Theorem~\ref{wf-koszulity-implies-comodule-ext-isom-theorem}).

 Lemma~\ref{initially-projective-resolution-computes-ext} from
the appendix tells us that, for every $0\le i\le n$ and any left
$C^*$\+module $M$, the space $\Ext^i_{C^*}(k,M)$ can be computed
as the degree~$i$ cohomology space of the complex obtained by applying
the functor $\Hom_{C^*}({-},M)$ to
the resolution~\eqref{dual-initially-less-projective-resolution} of
the left $C^*$\+module~$k$.
 For $M=k$, the resulting complex turns into
\begin{multline} \label{computing-ext-by-less-init-proj-for-k}
 0\lrarrow V_0\lrarrow V_1\lrarrow\dotsb\lrarrow V_{n-1}\\
 \lrarrow\Hom_{C^*}((V_n\ot_k C)^*,\>k)\lrarrow
 \Hom_{C^*}((V_{n+1}\ot_k C)^*,k).
\end{multline}
 By Lemma~\ref{hom-from-dual-into-k},
the complex~\eqref{computing-ext-by-less-init-proj-for-k} is
naturally isomorphic to the complex
\begin{equation} \label{double-dual-zero-differential}
 0\lrarrow V_0\lrarrow V_1\lrarrow\dotsb\lrarrow V_{n-1}\\
 \lrarrow V_n^{**}\lrarrow V_{n+1}^{**}
\end{equation}
obtained by applying the functor $(({-})_\gamma)^{**}$ to
the coresolution~\eqref{minimal-coresolution}.
 As the latter coresolution was chosen to be minimal,
the differential in the complex~\eqref{double-dual-zero-differential}
(or, which is the same, in
the complex~\eqref{computing-ext-by-less-init-proj-for-k}) vanishes.
 Thus $\Ext^n_{C^*}(k,k)\simeq V_n^{**}\simeq\Ext^n_C(k,k)^{**}$.

 This proves existence of
\emph{an} isomorphism~\eqref{Ext-n-double-dual-isomorphism}.
 To show that this isomorphism is \emph{natural}, consider
an arbitrary injective coresolution
\begin{equation} \label{arbitrary-injective-coresolution}
 0\lrarrow k\lrarrow J^0\lrarrow J^1\lrarrow J^2\lrarrow\dotsb
\end{equation}
of the right $C$\+comodule~$k$.
 Then
\begin{equation} \label{dual-nonprojective-resolution}
 0\llarrow k\llarrow (J^0)^*\llarrow(J^1)^*\llarrow(J^2)^*
 \llarrow\dotsb
\end{equation}
is a (nonprojective) resolution of the left $C^*$\+module~$k$.
 However, the coresolution~\eqref{arbitrary-injective-coresolution}
is homotopy equivalent (as a complex of injective right $C$\+comodules)
to the minimal coresolution~\eqref{minimal-coresolution}.
 Hence the resolution~\eqref{dual-nonprojective-resolution} is homotopy
equivalent (as a complex of left $C^*$\+modules) to the initially
projective resolution~\eqref{dual-initially-less-projective-resolution}.
 Consequently, the natural map~\eqref{resolution-to-ext-comparison}
from Lemma~\ref{initially-projective-resolution-computes-ext} for
the nonprojective resolution~\eqref{dual-nonprojective-resolution} is
naturally isomorphic to the similar map for the initially projective
resolution~\eqref{dual-initially-less-projective-resolution}.

 Thus one can compute $\Ext^i_{C^*}(k,k)$ for $0\le i\le n$
using the nonprojective
resolution~\eqref{dual-nonprojective-resolution}, obtaining
natural isomorphisms
$$
 \Ext^i_{C^*}(k,k)\simeq H^i\Hom_{C^*}((J^\bu)^*,k)
 \simeq H^i((J^\bu_\gamma)^{**})\simeq(H^i(J^\bu_\gamma))^{**}
 \simeq \Ext^i_C(k,k)^{**}
$$
(with the second isomorphism provided by
Lemma~\ref{hom-from-dual-into-k}) for all $0\le i\le n$.
\end{proof}

\begin{proof}[Proof of
Theorem~\ref{comodule-ext-isom-implies-wf-koszulity-theorem}]
 The map $\Ext^n_C(L,M)\rarrow\Ext^n_{C^*}(L,M)$ is
injective for all left $C$\+comodules $L$ and $M$ in our assumptions
by Corollary~\ref{comodule-ext-injective-map-cor}.
 The dimension cardinality assertion of the theorem is provided by
Proposition~\ref{comodule-ext-1-isom-implies-prop} (for $n=1$)
or Proposition~\ref{comodule-ext-n-isom-implies-prop} (for $n\ge2$)
together with Lemma~\ref{dimension-of-the-dual-lemma}.
 The nonsurjectivity assertion follows from the dimension inequality.
\end{proof}

\begin{rem}
 Arguing more carefully, one can show that, in the assumptions of
Proposition~\ref{comodule-ext-n-isom-implies-prop},
the isomorphism~\eqref{Ext-n-double-dual-isomorphism} identifies
the map $\Ext^n_C(k,k)\rarrow\Ext^n_{C^*}(k,k)$ induced by
the comodule inclusion functor $\Upsilon\:C\Comodl\rarrow C^*\Modl$
with the natural inclusion $\Ext^n_C(k,k)\rarrow
\Ext^n_{C^*}(k,k)^{**}$ of a vector space $\Ext^n_C(k,k)$ into its
double dual vector space.
 For this purpose, one can start with comparing the proofs of
Propositions~\ref{ext-k-k-co-contra-prop}
and~\ref{comodule-ext-n-isom-implies-prop} in order to observe that
the forgetful functor $\Theta\:C\Contra\rarrow C^*\Modl$ induces
an isomorphism $\Ext^{C,n}(k,k)\simeq\Ext_{C^*}^n(k,k)$ in
the assumptions of Proposition~\ref{comodule-ext-n-isom-implies-prop}.
 Then the result of Proposition~\ref{ext-k-k-co-contra-prop} identifies
the map $\Ext^n_{C^\rop}(k,k)\rarrow\Ext^n_{C^*}(k,k)$ induced by
the dual vector space functor $N\longmapsto N^*\:\Comodr C\rarrow
C^*\Modl$ with the natural inclusion $\Ext^n_{C^\rop}(k,k)\rarrow
\Ext^n_{C^\rop}(k,k)^{**}$.
 It remains to restrict the consideration to finite-dimensional
comodules and use the observations from the second proof of
Proposition~\ref{ext-k-k-left-right-symmetric}.
\end{rem}

\Section{Weakly Finite Koszulity Implies Contramodule Ext Isomorphism}
\label{wf-koszulity-implies-contramodule-ext-isom-secn}

 Before stating the main theorem of this section, let us have
a discussion of \emph{separated} and \emph{nonseparated} contramodules.

 Let $C=(C,\gamma)$ be a coaugmented coalgebra and $P$ be a left
$C$\+contramodule.
 For every $m\ge0$, denote by $F^mP\subset P$ the image of the subspace
$\Hom_k(C/F_{m-1}C,P)\subset\Hom_k(C,P)$ under the contraaction
map $\pi\:\Hom_k(C,P)\rarrow P$.
 So one has $F^0P=P$ and $F^1P=\ker(P\twoheadrightarrow{}^\gamma\!P)$
(in the notation of Sections~\ref{conilpotent-minimal-secn}
and~\ref{wf-koszulity-implies-comodule-ext-isom-secn}).
 One can check that $F$ is a contramultiplicative decreasing filtration
on $P$ compatible with the increasing filtration $F$ on $C$, that is
$$
 \pi(\Hom_k(C/F_{q-1}C,F^pP))\subset F^{p+q}P
$$
for all $p$, $q\ge0$.

 In other words, this means that $F^mP$ is a $C$\+subcontramodule of $P$
for every $m\ge0$ \emph{and} the successive quotient $C$\+contramodules
$F^mP/F^{m+1}P$ are trivial (i.~e., the contraaction of $C$ in them
is induced by the coaugmentation~$\gamma$).
 In fact, one has $F^mP/F^{m+1}P={}^\gamma(F^mP)$ for every $m\ge0$.

 The problem with the decreasing filtration $F$ on $P$ is that it is
\emph{not} in general separated (not even when $C$ is a conilpotent
coalgebra); see the counterexample in~\cite[Section~1.5]{Prev}
(going back to~\cite[Example~2.5]{Sim}, \cite[Example~3.20]{Yek},
and~\cite[Section~A.1.1]{Psemi}).
 The next lemma explains how this problem can be partially rectified
(for a conilpotent coalgebra~$C$).

 We say that a vector space $P$ endowed with a decreasing filtration $F$
is \emph{separated} if the natural map to the projective limit
$$
 \lambda_{P,F}\:P\lrarrow\varprojlim\nolimits_{m\ge1} P/F^mP
$$
is injective, and that $P$ is \emph{complete} if the map~$\lambda_{P,F}$
is surjective.
 A contramodule $P$ over a conilpotent coalgebra $C$ is said to be
\emph{separated} (respectively, \emph{complete}) if it is separated
(resp., complete) with respect to its natural decreasing filtration $F$
constructed above.

 Notice that, for any subcontramodule $Q$ in a contramodule $P$, one has
$F^mQ\subset F^mP$ for all $m\ge0$.
 Thus any subcontramodule of a separated contramodule is separated.
 For a free contramodule $P=\Hom_k(C,V)$, one has
$F^mP=\Hom_k(C/F_{m-1}C,V)$.
 Hence all free (and therefore, all projective) contramodules over
a conilpotent coalgebra are separated and complete.
{\emergencystretch=1em\par}

\begin{lem} \label{complete-separated-contramodules}
 Let $C$ be a conilpotent coalgebra.  Then \par
\textup{(a)} any $C$\+contramodule is complete; \par
\textup{(b)} any $C$\+contramodule can be presented as the quotient
contramodule of a separated contramodule by a separated subcontramodule.
\end{lem}

\begin{proof}
 Part~(b) holds because any contramodule is a quotient of
a free contramodule, which is separated, and any subcontramodule of
a separated contramodule is separated.
 For part~(a), see~\cite[Lemma~A.2.3]{Psemi}.
\end{proof}

 The aim of this section is to prove the following theorem.

\begin{thm} \label{wf-koszulity-implies-contramodule-ext-isom-theorem}
 Let $C$ be a conilpotent coalgebra over a field~$k$ and $n\ge1$
be an integer such that the vector space\/ $\Ext^i_C(k,k)$ is
finite-dimensional for all\/ $1\le i\le n$.
 Then the map of the Ext spaces
\begin{equation} \label{contramodule-ext-map-again}
 \Ext^{C,i}(P,Q)\lrarrow\Ext_{C^*}^i(P,Q)
\end{equation}
induced by the forgetful functor\/ $\Theta\:C\Contra\rarrow C^*\Modl$ is
an isomorphism for all left $C$\+contramodules $P$, all \emph{separated}
left $C$\+contramodules $Q$, and all\/ $0\le i\le n-1$.
 The map~\eqref{contramodule-ext-map-again} is also an isomorphism
for \emph{all} left $C$\+contramodules $P$ and $Q$ and all\/
$0\le i\le n-2$.
\end{thm}

 The following lemma, describing some tensor products of $C^*$\+modules
for a coalgebra~$C$, is a partial dual version of
Lemma~\ref{C-star-hom-described}.

\begin{lem} \label{C-star-tensor-product-described}
 Let $C$ be a coalgebra over $k$, let $N$ be a left $C$\+comodule,
and let $U$ be a $k$\+vector space.
 Then there is a natural map of $k$\+vector spaces
$$
 \zeta_{N,U}\:N^*\ot_{C^*}\Hom_k(C,U)\lrarrow
 \Hom_k(N,U)
$$
which is an isomorphism whenever $N$ is a finitely cogenerated injective
$C$\+comodule.
 Here the left $C^*$\+module structure on\/ $\Hom_k(C,U)$ comes from
the free left $C$\+con\-tra\-mod\-ule structure, or equivalently,
is induced by the right $C^*$\+module structure on $C$ (coming from
the right $C$\+comodule structure on~$C$).
\end{lem}

\begin{proof}
 The map~$\zeta_{N,U}$ is defined by the formula
$$
 \zeta_{N,U}(f\ot g)(x)=f(x_{(0)})g(x_{(-1)})
$$
for all $f\in N^*$, \ $g\in\Hom_k(C,U)$, and $x\in N$.
 Here $\nu(x)=x_{(-1)}\ot x_{(0)}\in C\ot_k N$ is a notation for
the left $C$\+coaction map $\nu\:N\rarrow C\ot_k N$.
 So $f(x_{(0)})\in k$ and $g(x_{(-1)})\in U$.

 To prove the isomorphism assertion for a finitely cogenerated
injective $C$\+comodule $N$, it suffices to consider the case
$N=C\ot_kV$, where $\dim_kV<\infty$.
 In this case, $N^*=V^*\ot_k C^*$ is a free right $C^*$\+module,
so $N^*\ot_{C^*}\Hom_k(C,U)=V^*\ot_k\Hom_k(C,U)=\Hom_k(V,\Hom_k(C,U))
=\Hom_k(C\ot_kV,\>U)$.
\end{proof}

 The most difficult part of the proof of
Theorem~\ref{wf-koszulity-implies-contramodule-ext-isom-theorem}
is the case of $n=1$.
 This is a previously known result~\cite{Psm}.
 Notice that, according to the following theorem, no separatedness
assumptions are needed in
Theorem~\ref{wf-koszulity-implies-contramodule-ext-isom-theorem}
for $n=1$.

\begin{thm} \label{fully-faithful-contramodule-forgetful}
 Let $C$ be a finitely cogenerated conilpotent coalgebra over
a field~$k$.
 Then the contramodule forgetful functor\/ $\Theta\:C\Contra\rarrow
C^*\Modl$ is fully faithful.
\end{thm}

\begin{proof}
 Moreover, the forgetful functor $C\Contra\rarrow R\Modl$ (defined
as the composition $C\Contra\rarrow C^*\Modl\rarrow R\Modl$) is fully
faithful for any dense subring $R\subset C^*$ in the natural
pro-finite-dimensional (otherwise known as linearly compact or
pseudocompact) topology of the $k$\+algebra~$C^*$.
 See~\cite[Theorem~2.1]{Psm}.
 For a discussion of related results with further references,
see~\cite[Section~3.8]{Prev}.
\end{proof}

\begin{proof}[Proof of
Theorem~\ref{wf-koszulity-implies-contramodule-ext-isom-theorem}]
 The argument resembles the proofs of the contramodule Ext comparison
theorems in~\cite[Sections~5.4 and~5.7]{Prel} (see
specifically~\cite[Theorem~5.20]{Prel}).
 Notice first of all that the coalgebra $C$ is finitely cogenerated
in the assumptions of
Theorem~\ref{wf-koszulity-implies-contramodule-ext-isom-theorem},
so the functor $\Theta$ is fully faithful by
Theorem~\ref{fully-faithful-contramodule-forgetful}.

 To prove the first assertion of the theorem, let $Q$ be a separated
left $C$\+con\-tra\-mod\-ule.
 According to Lemma~\ref{ext-preservation-enough-for-projectives}, it
suffices to show that $\Ext^i_{C^*}(P,Q)=0$ for all projective left
$C$\+contramodules $P$ and all integers $1\le i\le n-1$.

 Lemma~\ref{complete-separated-contramodules}(a) says that
the $C$\+contramodule $Q$ is complete with respect
to its natural decreasing filtration~$F$; so it is separated and
complete under our assumptions.
 According to the discussion preceding the lemma, the successive
quotient contramodules $F^mQ/F^{m-1}Q$ have trivial $C$\+contramodule
structures (i.~e., their $C$\+contramodule structures are induced by
the coaugmentation~$\gamma$).
 Using the dual Eklof lemma~\cite[Proposition~18]{ET}, the question
reduces to showing that $\Ext_{C^*}^i(P,T)=0$ for all $k$\+vector spaces
$T$ with trivial $C$\+contramodule structures.

 Any vector space $T$ is a direct summand of the dual vector space $W^*$
to some vector space~$W$.
 Thus it suffices to show that the vector space
$\Ext_{C^*}^i(P,W^*)\simeq\Tor^{C^*}_i(W,P)^*$ vanishes.
 Finally, $W$ is a direct sum of copies of the one-dimensional
vector space~$k$, and Tor commutes with the direct sums.
 So we need to show that $\Tor^{C^*}_i(k,P)=0$ for any projective left
$C$\+contramodule $P$ and $1\le i\le n-1$ (and the trivial
right $C^*$\+module structure on~$k$).
 Equivalently, this means that the space $\Tor^{C^*}_i(k,\Hom_k(C,U))$
should vanish for all $k$\+vector spaces $U$ and $1\le i\le n-1$.

 This time, we consider a minimal injective/cofree coresolution of
the left $C$\+comod\-ule~$k$,
\begin{multline} \label{left-comodule-minimal-coresolution}
 0\lrarrow k\lrarrow C\ot_k V_0\lrarrow C\ot_k V_1\lrarrow C\ot_k V_2
 \lrarrow\dotsb \\ \lrarrow C\ot_k V_{n-1}\lrarrow C\ot_k V_n
 \lrarrow C\ot_k V_{n+1}\lrarrow\dotsb,
\end{multline}
and apply the dual vector space functor $\Hom_k({-},k)$ to it in order
to obtain an initially projective resolution of the right
$C^*$\+module~$k$,
\begin{multline} \label{dual-initially-projective-resolution-other-side}
 0\llarrow k\llarrow V_0^*\ot_kC^*\llarrow V_1^*\ot_kC^*\llarrow
 V_2^*\ot_kC^* \llarrow\dotsb \\
 \llarrow V_{n-1}^*\ot_kC^*\llarrow V_n^*\ot_kC^*
 \llarrow (C\ot_k V_{n+1})^*\llarrow\dotsb
\end{multline}
 Here the vector space $V_i\simeq\Ext_C^i(k,k)$ is finite-dimensional
for $0\le i\le n$ by assumption.
 For every $0\le i\le n+1$ and any left $C^*$\+module $P$, the space
$\Tor^{C^*}_i(k,P)$ can be computed as the degree~$i$ homology space
of the complex obtained by applying the functor ${-}\ot_{C^*}P$ to
the resolution~\eqref{dual-initially-projective-resolution-other-side}
of the right $C^*$\+module~$k$ (see
Lemma~\ref{initially-projective-resolution-computes-tor}).
 We are only interested in $0\le i\le n-1$ now, so we only write down
the resulting complex in the cohomological degrees from~$0$ to~$n$.
 It has the form
$$
 0\llarrow V_0^*\ot_kP\llarrow V_1^*\ot_kP\llarrow\dotsb
 \llarrow V_{n-1}^*\ot_kP\llarrow V_n^*\ot_kP,
$$
which for $P=\Hom_k(C,U)$ turns into
\begin{multline} \label{computing-tor-by-init-proj-for-free}
 0\llarrow\Hom_k(V_0,\Hom_k(C,U))\llarrow\Hom_k(V_1,\Hom_k(C,U))
 \llarrow\dotsb \\ \llarrow\Hom_k(V_{n-1},\Hom_k(C,U))\llarrow
 \Hom_k(V_n,\Hom_k(C,U)).
\end{multline}

 By Lemma~\ref{C-star-tensor-product-described},
the complex~\eqref{computing-tor-by-init-proj-for-free}
(in the cohomological degrees~$\le n$) can be simply obtained by
applying the contravariant vector space Hom functor $\Hom_k({-},U)$
to the coresolution~\eqref{left-comodule-minimal-coresolution}.
 Consequently, the complex~\eqref{computing-tor-by-init-proj-for-free}
is exact in the cohomological degrees $0<i\le n-1$.
 Therefore, $\Tor^{C^*}_i(k,\Hom_k(C,U))=0$ for $0<i\le n-1$, as desired.

 Having proved the first assertion of the theorem, we can now easily
deduce the second one.
 For a nonseparated left $C$\+contramodule~$Y$,
by Lemma~\ref{ext-preservation-enough-for-projectives} we only need
to show that $\Ext^i_{C^*}(P,Y)=0$ for all projective left
$C$\+contramodules $P$ and $1\le i\le n-2$.
 Lemma~\ref{complete-separated-contramodules}(b) provides a short
exact sequence $0\rarrow Q'\rarrow Q''\rarrow Y\rarrow0$ with separated
$C$\+contramodules $Q'$ and~$Q''$.
 In the long exact sequence $\dotsb\rarrow\Ext^i_{C^*}(P,Q'')\rarrow
\Ext^i_{C^*}(P,Y)\rarrow\Ext^{i+1}_{C^*}(P,Q')\rarrow\dotsb$ we have
$\Ext^i_{C^*}(P,Q'')=0=\Ext^{i+1}_{C^*}(P,Q')$, hence
$\Ext^i_{C^*}(P,Y)=0$.
\end{proof}

\begin{cor} \label{contramodule-ext-injective-map-cor}
 Let $C$ be a conilpotent coalgebra over a field~$k$ and $n\ge0$ be
an integer such that the vector space\/ $\Ext_C^i(k,k)$ is
finite-dimensional for all\/ $1\le i\le n$.
 Then the map of the Ext spaces
\begin{equation} \label{contramodule-ext-map-2nd-repeat}
 \Ext^{C,i}(P,Q)\lrarrow\Ext_{C^*}^i(P,Q)
\end{equation}
induced by the forgetful functor\/ $\Theta\:C\Contra\rarrow C^*\Modl$
is injective for all left $C$\+contramodules $P$, all \emph{separated}
left $C$\+contramodules $Q$, and all\/ $0\le i\le n$.
 The map~\eqref{contramodule-ext-map-2nd-repeat} is also injective
for \emph{all} left $C$\+contramodules $P$ and $Q$ and all\/
$0\le i\le n-1$.
\end{cor}

\begin{proof}
 No separatedness assumption is needed in the obvious case $n=0$.
 For $n\ge1$, this is a purely formal consequence of
Theorem~\ref{wf-koszulity-implies-contramodule-ext-isom-theorem}.
 See Lemma~\ref{induced-on-Ext-1-lemma} for the case $n=1$
(when the separatedness assumption is not needed, either, by virtue
of Theorem~\ref{fully-faithful-contramodule-forgetful}) and
Lemma~\ref{next-degree-ext-injectivity-lemma}(a) for the general case.
\end{proof}

\Section{Contramodule Ext Isomorphism Implies Weakly Finite Koszulity}
\label{contramodule-ext-isom-implies-wf-koszulity-secn}

 The aim of this section is to prove the following theorem.
 Together with the results of the previous
Sections~\ref{wf-koszulity-implies-comodule-ext-isom-secn}\+-%
\ref{wf-koszulity-implies-contramodule-ext-isom-secn}, this
will allow us to finish the proof of
Theorem~\ref{cohomological-degree-careful-main-theorem}
from the introduction.

\begin{thm} \label{contramodule-ext-isom-implies-wf-koszulity-theorem}
 Let $C$ be a conilpotent coalgebra over a field~$k$ and $n\ge1$ be
the minimal integer for which the vector space\/ $\Ext^n_C(k,k)$ is
infinite-dimensional.
 Let $T$ be an infinite-dimensional $k$\+vector space endowed with
the trivial left $C$\+con\-tra\-mod\-ule structure (provided by
the coaugmentation $\gamma\:C\rarrow k$).
 Then the map\/ $\Ext^{C,n}(T,k)\rarrow\Ext_{C^*}^n(T,k)$ induced by
the contramodule forgetful functor\/ $\Theta\:C\Contra\rarrow
C^*\Modl$ is \emph{not} injective.
 Consequently, there exists a projective left $C$\+contramodule $P$ such
that the map\/ $\Ext^{C,n-1}(P,k)\rarrow\Ext_{C^*}^{n-1}(P,k)$ induced
by the functor\/ $\Theta$ is (injective, but) \emph{not} surjective.
\end{thm}

 We start with a discussion of the case $n=1$ in
Theorem~\ref{contramodule-ext-isom-implies-wf-koszulity-theorem}.
 This is the only case when $\Ext^{C,n-1}(P,k)\ne0$.

\begin{ex} \label{contramodule-extension-example}
 Let $C$ be an infinitely cogenerated conilpotent coalgebra over~$k$,
and let $T$ be an infinite-dimensional vector space endowed with
the trivial left $C$\+contramodule structure.
 Then we claim that there exists a nonsplit short exact sequence
of left $C$\+contramodules
$$
 0\lrarrow k\lrarrow Q\lrarrow T\lrarrow0
$$
that splits as a short exact sequence of left $C^*$\+modules.
 Consequently, the splitting map $q\:Q\rarrow k$ is a morphism of
$C^*$\+modules but \emph{not} a morphism of $C$\+con\-tra\-mod\-ules.
 So the forgetful functor $\Theta$ is not full.

 To construct the desired short exact sequence, put $C_+=C/\gamma(k)$
and $V=\ker(C_+\to C_+\ot_k C_+)=\Ext^1_C(k,k)$, as in
Example~\ref{noncomodule-extension-example}.
 So we have $F_0C=\gamma(k)$ and $F_1C=F_0C\oplus V$.
 Notice that $V^*\ot_kT$ is naturally a proper vector subspace in
$\Hom_k(V,T)$, as both the vector spaces $V$ and $T$ are
infinite-dimensional.
 Essentially, $V^*\ot_kT\subset\Hom_k(V,T)$ is the subspace of all
linear maps $V\rarrow T$ of finite rank.
 Choose a linear function $\phi\:\Hom_k(V,T)\rarrow k$ vanishing
on $V^*\ot_kT$.

 Let the underlying vector space of $Q$ be simply the direct sum
$k\oplus T$.
 Define the contraaction map $\pi\:\Hom_k(C,Q)\rarrow Q$ as follows.
 The component $\pi_{k,T}\:\Hom_k(C,k)\rarrow T$ of the map~$\pi$
is zero.
 The component $\pi_{k,k}\:\Hom_k(C,k)\rarrow k$ is induced by
the coaugmentation~$\gamma$.
 The component $\pi_{T,T}\:\Hom_k(C,T)\rarrow T$ is also induced
by~$\gamma$; specifically, it is defined as the composition
$\Hom_k(C,T)\rarrow\Hom_k(k,T)\rarrow T$ of the surjective map
$\Hom_k(\gamma,T)\:\Hom_k(C,T)\rarrow\Hom_k(k,T)$ and the identity
isomorphism $\Hom_k(k,T)\simeq T$.

 Finally, the component $\pi_{T,k}\:\Hom_k(C,T)\rarrow k$ of
the map~$\pi$ is defined as the composition
$\Hom_k(C,T)\rarrow\Hom_k(F_1C,T)\rarrow\Hom_k(V,T)\rarrow k$ of
the surjective map $\Hom_k(C,T)\rarrow\Hom_k(F_1C,T)$ induced
by the inclusion $F_1C\rarrow C$, the direct summand projection
$\Hom_k(F_1C,T)=\Hom_k(F_0C\oplus V,\>T)\rarrow\Hom_k(V,T)$,
and the chosen linear function $\phi\:\Hom_k(V,T)\rarrow k$.

 Then the direct sum decomposition of vector spaces $Q=k\oplus T$
still holds as a direct sum decomposition in the module category
$C^*\Modl$, since the composition of the embedding $C^*\ot_k T
\rarrow\Hom_k(C,T)$ with the map $\pi_{T,k}\:\Hom_k(C,T)\rarrow k$
is the zero map.
 But the $C$\+contramodule structure on $Q$ is nontrivial (not
induced by the coaugmentation~$\gamma$), so $Q$ is \emph{not}
isomorphic to $k\oplus T$ in $C\Contra$.
\end{ex}

 In order to pass from some object $Q$ in
Example~\ref{contramodule-extension-example} to a projective object $P$
in the context
Theorem~\ref{contramodule-ext-isom-implies-wf-koszulity-theorem}
for $n=1$, we will use
Lemma~\ref{full-and-faithfulness-enough-for-projectives} from
the appendix.

 In the case $n>1$, we need the following simple
coalgebra-theoretic lemma.

\begin{lem} \label{hom-from-hom-into-k}
 Let $(C,\gamma)$ be a coaugmented coalgebra, $N$ be a right
$C$\+comodule, and $T$ be a $k$\+vector space.
 Then there is a natural isomorphism of $k$\+vector spaces
$$
 \Hom^C(\Hom_k(N,T),k)\simeq\Hom_k(N_\gamma,T)^*.
$$
\end{lem}

\begin{proof}
 One has ${}^\gamma\!\Hom_k(N,T)\simeq\Hom_k(N_\gamma,T)$ by
formula~\eqref{E-quotcontramodule-of-Hom-contramodule} from
Section~\ref{conilpotent-minimal-secn}.
 Hence
$$
 \Hom^C(\Hom_k(N,T),k)\simeq\Hom_k({}^\gamma\!\Hom_k(N,T),k)
 \simeq\Hom_k(\Hom_k(N_\gamma,T),k)
$$
as desired (cf.\ Lemma~\ref{hom-from-dual-into-k}).
\end{proof}

\begin{prop} \label{contramodule-ext-n-monom-implies-prop}
 Let $C$ be a finitely cogenerated conilpotent coalgebra over
a field~$k$ and $n\ge2$ be the minimal integer for which
the vector space\/ $\Ext^n_C(k,k)$ is infinite-dimensional.
 Let $T$ be an infinite-dimensional $k$\+vector space endowed with
the trivial left $C$\+con\-tra\-mod\-ule structure (provided by
the coaugmentation~$\gamma$).
 Then the map\/ $\Ext^{C,n}(T,k)\rarrow\Ext_{C^*}^n(T,k)$ induced by
the contramodule forgetful functor\/ $\Theta\:C\Contra\rarrow
C^*\Modl$ is surjective, but \emph{not} injective.
 Consequently, one has $\Ext_{C^*}^{n-1}(P,k)\ne0$ for
the projective/free left $C$\+contramodule $P=\Hom_k(C,T)$.
\end{prop}

\begin{proof}
 Our strategy of proving the first assertion of the proposition is
to explicitly compute the map of Ext spaces in question, and see that
it is surjective but not injective.
 As in the proof of
Theorem~\ref{wf-koszulity-implies-comodule-ext-isom-theorem}, we choose
a minimal injective/cofree coresolution~\eqref{minimal-coresolution}
of the right $C$\+comodule~$k$,
\begin{multline} \label{minimal-coresolution-again}
 0\lrarrow k\lrarrow V_0\ot_kC\lrarrow V_1\ot_kC\lrarrow V_2\ot_kC
 \lrarrow\dotsb \\ \lrarrow V_{n-1}\ot_kC\lrarrow V_n\ot_kC
 \lrarrow V_{n+1}\ot_kC \lrarrow\dotsb
\end{multline}
 The vector space $V_i\simeq\Ext^i_{C^\rop}(k,k)\simeq\Ext^i_C(k,k)$
is finite-dimensional for $0\le i\le n-1$ and infinite-dimensional
for $i=n$ by assumption.

 Applying the contravariant vector space Hom functor $\Hom_k({-},T)$
to the complex~\eqref{minimal-coresolution-again}, we obtain
a minimal projective resolution of the left $C$\+contramodule~$T$,
\begin{multline} \label{minimal-contramodule-projective-resolution}
 0\larrow T\larrow\Hom_k(C,\Hom_k(V_0,T))\larrow
 \Hom_k(C,\Hom_k(V_1,T))\larrow\dotsb \\
 \larrow \Hom_k(C,\Hom_k(V_{n-1},T)) \larrow \Hom_k(C,\Hom_k(V_n,T))
 \\ \larrow\Hom_k(C,\Hom_k(V_{n+1},T))\larrow\dotsb
\end{multline}
 On the other hand, applying the dual vector space functor
$\Hom_k({-},k)$ to
the complex~\eqref{minimal-coresolution-again}, we obtain
a resolution of the left $C^*$\+module~$k$
\begin{multline} \label{dual-init-proj-left-module-k-resolution}
 0\llarrow k\llarrow C^*\ot_kV_0^*\llarrow C^*\ot_kV_1^*\llarrow
 C^*\ot_kV_2^* \llarrow\dotsb \\
 \llarrow C^*\ot_kV_{n-1}^*\llarrow (V_n\ot_k C)^*
 \llarrow (V_{n+1}\ot_k C)^*\llarrow\dotsb
\end{multline}
as in formula~\eqref{dual-initially-less-projective-resolution} from
the proof of Proposition~\ref{comodule-ext-n-isom-implies-prop}.
 Applying the vector space tensor product functor ${-}\ot_k T$
to the complex~\eqref{dual-init-proj-left-module-k-resolution}, we
get an initially projective resolution of the left $C^*$\+module~$T$,
\begin{multline} \label{dual-init-proj-left-module-T-resolution}
 0\larrow T\llarrow C^*\ot_kV_0^*\ot_kT\larrow C^*\ot_kV_1^*\ot_kT
 \larrow\dotsb \\ \larrow C^*\ot_kV_{n-1}^*\ot_kT
 \larrow (V_n\ot_k C)^*\ot_kT 
 \larrow (V_{n+1}\ot_k C)^*\ot_kT.
\end{multline}

 For any left $C$\+contramodule $Y$, the projective
resolution~\eqref{minimal-contramodule-projective-resolution} of
the trivial left $C$\+contramodule $T$ can be used in order to
compute the vector spaces $\Ext^{C,i}(T,Y)$ in all cohomological
degrees $i\ge0$.
 On the other hand,
the resolution~\eqref{dual-init-proj-left-module-T-resolution} of
the trivial left $C^*$\+module~$T$ is only initially projective.
 By Lemma~\ref{initially-projective-resolution-computes-ext}, for
any left $C^*$\+module $Y$,
the resolution~\eqref{dual-init-proj-left-module-T-resolution} can be
used in order to compute the vector spaces $\Ext_{C^*}^i(T,Y)$ in
the cohomological degrees $0\le i\le n$.

 Now we have a natural injective morphism of complexes of
left $C^*$\+modules from
the complex~\eqref{dual-init-proj-left-module-T-resolution}
to the complex~\eqref{minimal-contramodule-projective-resolution},
acting by the identity map on the leftmost terms~$T$.
 This morphism of resolutions can be used in order to compute
the maps $\Ext^{C,i}(T,Y)\lrarrow\Ext_{C^*}^i(T,Y)$ in
the cohomological degrees $0\le i\le n$ for any left
$C$\+contramodule~$Y$.

 Put $Y=k$.
 Applying the contramodule Hom functor $\Hom^C({-},k)$ to the projective
resolution~\eqref{minimal-contramodule-projective-resolution} of
the trivial left $C$\+contramodule~$T$, we obtain a complex
\begin{multline} \label{computing-ext-by-minimal-contramodule-proj}
 0\lrarrow\Hom_k(V_0,T)^*\lrarrow\Hom_k(V_1,T)^*\lrarrow\dotsb \\
 \lrarrow\Hom_k(V_{n-1},T)^* \lrarrow \Hom_k(V_n,T)^*
 \lrarrow\Hom_k(V_{n+1},T)^*\lrarrow\dotsb
\end{multline}
 By Lemma~\ref{hom-from-hom-into-k},
the complex~\eqref{computing-ext-by-minimal-contramodule-proj}
is naturally isomorphic to the complex obtained by applying
the functor $\Hom_k(({-})_\gamma,T)^*$ to
the coresolution~\eqref{minimal-coresolution-again}.
 As the latter coresolution was chosen to be minimal, the differential
in the complex~\eqref{computing-ext-by-minimal-contramodule-proj}
vanishes.

 Applying the module Hom functor $\Hom_{C^*}({-},k)$ to the initially
projective resolution~\eqref{dual-init-proj-left-module-T-resolution}
of the trivial left $C^*$\+module $T$, we obtain a complex
\begin{multline} \label{computing-ext-by-init-proj-of-T}
 0\lrarrow T^*\ot_k V_0\lrarrow T^*\ot_k V_1\lrarrow\dotsb
 \lrarrow T^*\ot_k V_{n-1} \\
 \lrarrow\Hom_k(T,\Hom_{C^*}((V_n\ot_k C)^*,k))
 \lrarrow\Hom_k(T,\Hom_{C^*}((V_{n+1}\ot_k C)^*,k)).
\end{multline}
 The complex~\eqref{computing-ext-by-init-proj-of-T} can be also
obtained by applying the covariant vector space Hom functor
$\Hom_k(T,{-})$ to the complex
\begin{multline} \label{computing-ext-by-init-proj-of-left-k}
 0\lrarrow V_0\lrarrow V_1\lrarrow\dotsb\lrarrow V_{n-1} \\
 \lrarrow\Hom_{C^*}((V_n\ot_k C)^*,k)
 \lrarrow\Hom_{C^*}((V_{n+1}\ot_k C)^*,k)
\end{multline}
produced by applying the module Hom functor $\Hom_{C^*}({-},k)$ to
the resolution~\eqref{dual-init-proj-left-module-k-resolution} of
the trivial $C^*$\+module~$k$
(see formula~\eqref{computing-ext-by-less-init-proj-for-k} in
the proof of Proposition~\ref{comodule-ext-n-isom-implies-prop}).

 By Lemma~\ref{hom-from-dual-into-k},
the complex~\eqref{computing-ext-by-init-proj-of-left-k} is
naturally isomorphic to the complex
\begin{equation} \label{double-dual-zero-differential-again}
 0\lrarrow V_0\lrarrow V_1\lrarrow\dotsb\lrarrow V_{n-1}\\
 \lrarrow V_n^{**}\lrarrow V_{n+1}^{**}
\end{equation}
obtained by applying the functor $(({-})_\gamma)^{**}$ to
the coresolution~\eqref{minimal-coresolution-again}.
 Once again, as the latter coresolution was chosen to be minimal,
the differential in
the complex~\eqref{double-dual-zero-differential-again}
(or, which is the same, in
the complex~\eqref{computing-ext-by-init-proj-of-left-k}) vanishes.
 We have computed the complex~\eqref{computing-ext-by-init-proj-of-T}
as having the form
\begin{multline} \label{further-computing-ext-by-init-proj-of-T}
 0\lrarrow\Hom_k(T,V_0)\lrarrow\Hom_k(T,V_1)\lrarrow\dotsb
 \lrarrow\Hom_k(T,V_{n-1}) \\
 \lrarrow\Hom_k(T,V_n^{**})
 \lrarrow\Hom_k(T,V_{n+1}^{**}),
\end{multline}
or, which is the same,
\begin{multline} \label{still-further-computing-ext-by-init-proj-of-T}
 0\lrarrow T^*\ot_k V_0\lrarrow T^*\ot_k V_1\lrarrow\dotsb
 \lrarrow T^*\ot_k V_{n-1} \\
 \lrarrow (V_n^*\ot_k T)^*
 \lrarrow (V_{n+1}^*\ot_k T)^*
\end{multline}
and zero differential.

 Now it is finally clear from
the formulas~\eqref{computing-ext-by-minimal-contramodule-proj}
and~\eqref{still-further-computing-ext-by-init-proj-of-T} that the map
\begin{equation} \label{ext-map-to-be-computed}
 \Ext^{C,n}(T,k)\lrarrow\Ext_{C^*}^n(T,k)
\end{equation}
is isomorphic to the natural map
\begin{equation} \label{ext-map-computed-as}
 \Hom(V_n,T)^*\lrarrow(V_n^*\ot_kT)^*,
\end{equation}
which can be obtained by applying the dual vector space functor
$\Hom_k({-},k)$ to the natural embedding
\begin{equation} \label{natural-embedding-to-be-dualized}
 V_n^*\ot_k T\lrarrow\Hom_k(V_n,T).
\end{equation}
 The map~\eqref{natural-embedding-to-be-dualized} is always injective,
but for infinite-dimensional vector spaces $V_n$ and $T$ it is not
surjective.
 Thus the map~\eqref{ext-map-computed-as}, and consequently the desired
map~\eqref{ext-map-to-be-computed}, is surjective but not injective.

 To deduce the second assertion of the proposition, we apply
Lemma~\ref{next-degree-ext-injectivity-lemma}(b) to the contramodule
forgetful functor $\Phi=\Theta$, the left $C$\+contramodule $Y=k$, and
the natural epimorphism of left $C$\+contramodules $P=\Hom_k(C,T)
\rarrow T$ with a kernel~$X$.
 By the first assertion of the proposition, there exists
a nonzero extension class $\beta\in\Ext^{C,n}(T,Y)$ annihilated by
the map~\eqref{ext-map-to-be-computed}.
 The class~$\beta$ is also annihilated by the map
$\Ext^{C,n}(T,Y)\rarrow\Ext^{C,n}(P,Y)=0$ (as $n\ge2$ and
the $C$\+contramodule $P$ is projective).
 Finally, the map $\Ext^{C,n-1}(X,Y)\rarrow\Ext_{C^*}^{n-1}(X,Y)$
is injective by Corollary~\ref{contramodule-ext-injective-map-cor}
(since the left $C$\+contramodule $Y=k$ is separated).
 So Lemma~\ref{next-degree-ext-injectivity-lemma}(b) tells us that
the map $0=\Ext^{C,n-1}(P,Y)\rarrow\Ext_{C^*}^{n-1}(P,Y)$ cannot
be surjective.
\end{proof}

\begin{proof}[Proof of
Theorem~\ref{contramodule-ext-isom-implies-wf-koszulity-theorem}]
 The case $n=1$ is covered by
Example~\ref{contramodule-extension-example} with
Lemma~\ref{full-and-faithfulness-enough-for-projectives}.
 The case $n\ge2$ is treated in
Proposition~\ref{contramodule-ext-n-monom-implies-prop}.
\end{proof}

\begin{proof}[Proof of
Theorem~\ref{cohomological-degree-careful-main-theorem}]
 Follows from
Theorems~\ref{wf-koszulity-implies-comodule-ext-isom-theorem},
\ref{comodule-ext-isom-implies-wf-koszulity-theorem},
\ref{wf-koszulity-implies-contramodule-ext-isom-theorem},
\ref{contramodule-ext-isom-implies-wf-koszulity-theorem}
(applied to the coalgebras $C$ and $C^\rop$)
and Propositions~\ref{ext-k-k-left-right-symmetric},
\ref{ext-k-k-co-contra-prop}.
 Specifically:

 (iv)\,$\Longrightarrow$\,(i) is
Theorem~\ref{wf-koszulity-implies-comodule-ext-isom-theorem};

 (i)\,$\Longrightarrow$\,(iv) is the first assertion of
Theorem~\ref{comodule-ext-isom-implies-wf-koszulity-theorem};

 (iv)\,$\Longrightarrow$\,(ii) is
Theorem~\ref{wf-koszulity-implies-comodule-ext-isom-theorem}
for the coalgebra $C^\rop$ together with
Proposition~\ref{ext-k-k-left-right-symmetric};

 (ii)\,$\Longrightarrow$\,(iv) is the first assertion of
Theorem~\ref{comodule-ext-isom-implies-wf-koszulity-theorem}
for the coalgebra $C^\rop$ together with
Proposition~\ref{ext-k-k-left-right-symmetric};

 (iv)\,$\Longrightarrow$\,(iii) is the first assertion of
Theorem~\ref{wf-koszulity-implies-contramodule-ext-isom-theorem};

 (iii)\,$\Longrightarrow$\,(iv) is the second assertion of
Theorem~\ref{contramodule-ext-isom-implies-wf-koszulity-theorem};

 (iv)\,$\Longleftrightarrow$\,(v) is
Proposition~\ref{ext-k-k-co-contra-prop}.

 The last assertion of
Theorem~\ref{cohomological-degree-careful-main-theorem} is provided
by the second assertion of
Theorem~\ref{wf-koszulity-implies-contramodule-ext-isom-theorem}.
\end{proof}

\Section{Half-Bounded Derived Full-and-Faithfulness}
\label{half-bounded-secn}

 The aim of this short section is to finish the proof of
Theorem~\ref{bounded-half-unbounded-derived-main-theorem} from
the introduction.
 The argument is based on the following result from
the preprint~\cite{Pper}.

\begin{prop} \label{equivalent-to-half-bounded-prop}
\textup{(a)} Let\/ $\sA$ and\/ $\sB$ be abelian categories, and\/
$\Upsilon\:\sB\rarrow\sA$ be a fully faithful exact functor.
 Assume that there are enough injective objects in the abelian
category\/ $\sA$, and the functor\/ $\Upsilon$ has a right adjoint
functor\/ $\Gamma\:\sA\rarrow\sB$.
 Then the induced triangulated functor between the bounded derived
categories\/ $\Upsilon^\bb\:\sD^\bb(\sB)\rarrow\sD^\bb(\sA)$ is fully
faithful if and only if the induced triangulated functor between
the bounded below derived categories\/ $\Upsilon^+\:\sD^+(\sB)\rarrow
\sD^+(\sA)$ is fully faithful. \par
\textup{(b)} Let\/ $\sA$ and\/ $\sB$ be abelian categories, and\/
$\Theta\:\sB\rarrow\sA$ be a fully faithful exact functor.
 Assume that there are enough projective objects in the abelian
category\/ $\sA$, and the functor\/ $\Theta$ has a left adjoint
functor\/ $\Delta\:\sA\rarrow\sB$.
 Then the induced triangulated functor between the bounded derived
categories\/ $\Theta^\bb\:\sD^\bb(\sB)\rarrow\sD^\bb(\sA)$ is fully
faithful if and only if the induced triangulated functor between
the bounded above derived categories\/ $\Theta^-\:\sD^-(\sB)\rarrow
\sD^-(\sA)$ is fully faithful.
\end{prop}

\begin{proof}
 Part~(b) is~\cite[Proposition~6.5\,(c)\,$\Leftrightarrow$\,(d)]{Pper}.
 The point is that any one of the functors $\Theta^\bb$ and
$\Theta^-$ is fully faithful if and only if the higher left derived
functors $\boL_n\Delta\:\sA\rarrow\sB$, \,$n\ge1$ of the functor
$\Delta$ vanish on the essential image of the functor~$\Theta$.
 Part~(a) is the dual assertion.
\end{proof}

\begin{proof}[Proof of
Theorem~\ref{bounded-half-unbounded-derived-main-theorem}]
 Follows from Theorem~\ref{cohomological-degree-careful-main-theorem}
and Proposition~\ref{equivalent-to-half-bounded-prop}.
 Specifically:

 (i)\,$\Longleftrightarrow$\,(vii) is
Theorem~\ref{cohomological-degree-careful-main-theorem}%
(i)\,$\Leftrightarrow$\,(iv);

 (iii)\,$\Longleftrightarrow$\,(vii) is
Theorem~\ref{cohomological-degree-careful-main-theorem}%
(ii)\,$\Leftrightarrow$\,(iv);

 (v)\,$\Longleftrightarrow$\,(vii) is
Theorem~\ref{cohomological-degree-careful-main-theorem}%
(iii)\,$\Leftrightarrow$\,(iv) and the last assertion of
Theorem~\ref{cohomological-degree-careful-main-theorem};

 (vii)\,$\Longleftrightarrow$\,(viii) is
Theorem~\ref{cohomological-degree-careful-main-theorem}%
(iv)\,$\Leftrightarrow$\,(v);

 (i)\,$\Longleftrightarrow$\,(ii) is
Proposition~\ref{equivalent-to-half-bounded-prop}(a);

 (iii)\,$\Longleftrightarrow$\,(iv) is
Proposition~\ref{equivalent-to-half-bounded-prop}(a);

 (v)\,$\Longleftrightarrow$\,(vi) is
Proposition~\ref{equivalent-to-half-bounded-prop}(b).

 In order to apply Proposition~\ref{equivalent-to-half-bounded-prop},
it only needs to be explained why the functor $\Upsilon\:C\Comodl
\rarrow C^*\Modl$ has a right adjoint and why the functor $\Theta\:
C\Contra\rarrow C^*\Modl$ has a left adjoint.
 In fact, these are quite general properties of coalgebras, as
the assumption that $C$ is conilpotent is not needed here.

 For any coalgebra $C$ over a field~$k$, the comodule inclusion
functor $\Upsilon\:C\Comodl\rarrow C^*\Modl$ has a right adjoint
functor $\Gamma\:C^*\Modl\rarrow C\Comodl$.
 The functor $\Gamma$ assigns to every left $C^*$\+module $N$ its
maximal submodule belonging to the essential image of
the functor~$\Upsilon$.
 In other words, $\Gamma(N)$ is the sum of all submodules of $N$
whose $C^*$\+module structure comes from a $C$\+comodule structure.
 Equivalently, $\Gamma(N)\subset N$ is the submodule consisting of
all the elements $x\in N$ whose annihilator ideals in $C^*$
(with respect to the action map $C^*\times N\rarrow N$) contain
the annihilator of some finite-dimensional vector subspace of $C$
(with respect to the pairing map $C^*\times C\rarrow k$)
\cite[Theorem~2.1.3(d)]{Swe}.

 For any coalgebra $C$ over a field~$k$, the contramodule forgetful
functor $\Theta\:C\Contra\allowbreak\rarrow C^*\Modl$ has a left
adjoint functor $\Delta\:C^*\Modl\rarrow C\Contra$.
 Given a left $C^*$\+module $M$, the adjunction morphism $M\rarrow
\Delta(M)$ is \emph{not} surjective in general; so $\Delta(M)$
\emph{cannot} be constructed as a quotient (contra)module of~$M$.

 To construct the functor $\Delta$, one defines it on \emph{free}
$C^*$\+modules by the rule $\Delta(C^*\ot_k\nobreak V)=\Hom_k(C,V)$
for every $k$\+vector space~$V$.
 Functoriality of the category object corepresenting a (given
corepresentable) functor provides a natural way to define the action
of $\Delta$ on morphisms of free $C^*$\+modules.
 Then there is always a unique way to extend an abelian category-valued
covariant additive functor defined on the full subcategory of free
modules over some ring $R$ to a right exact functor on the category of
all $R$\+modules.
 Simply put, to compute the contramodule $\Delta(M)$, one needs to
represent $M$ as the cokernel of a morphism of free $C^*$\+modules
$f\:F'\rarrow F''$; then $\Delta(M)$ is the cokernel of
the contramodule map $\Delta(f)\:\Delta(F')\rarrow\Delta(F'')$.
\end{proof}

\begin{rem} \label{second-kind-not-relevant-remark}
 Following the philosophy of the book~\cite{Psemi} and
the memoir~\cite{Pkoszul} (see also the discussion
in the survey~\cite[Section~7]{Pksurv}), one is generally supposed
to consider the \emph{coderived category} of $C$\+comodules
$\sD^\co(C\Comodl)$ and the \emph{contraderived category} of
$C$\+contramodules $\sD^\ctr(C\Contra)$, rather than the conventional
derived categories $\sD(C\Comodl)$ and $\sD(C\Contra)$.
 Let us point out, in this connection, that the difference between
the derived and the co/contraderived categories mostly \emph{does not
manifest itself} in the context of the present paper; certainly not
in the context of
Theorem~\ref{bounded-half-unbounded-derived-main-theorem}.
 The point is that the distinction between the conventional derived
and the co/contraderived categories of abelian categories is only
relevant for \emph{unbounded} complexes, while the results of
the present paper mostly concern the $\Ext$ spaces, which can be
computed in the \emph{bounded} derived categories.

 Specifically, let $\sA$ be an abelian category with exact functors
of infinite coproduct.
 Then the natural triangulated functor from the coderived to the derived
category $\sD^\co(\sA)\rarrow\sD(\sA)$ induces an equivalence of
the full subcategories of bounded below complexes, $\sD^{\co,+}(\sA)
\simeq\sD^+(\sA)$.
 Dually, if $\sB$ is an abelian category with exact functors of
infinite product, then the natural triangulated functor from
the contraderived to the derived category $\sD^\ctr(\sB)\rarrow\sD(\sB)$
induces an equivalence of the full subcategories of bounded above
complexes, $\sD^{\ctr,-}(\sB)\simeq\sD^-(\sB)$ \,\cite[Theorems~3.4.1
and~4.3.1]{Pkoszul}, \cite[Lemma~A.1.3]{Pcosh}.
 These references cover the case of co/contraderived categories in
the sense of Positselski (see~\cite[Section~7]{Pksurv} for
the terminology); for a similar result for coderived categories in
the sense of Becker, assume that there are enough injective objects
in $\sA$, denote the full subcategory of injective objects by
$\sE=\sA_\inj\subset\sA$, and refer to~\cite[Lemma~5.4]{PS2}.
 See also~\cite[Remark~4.1]{Psemi}.
\end{rem}

\Section{Co-Noetherian and Cocoherent Conilpotent Coalgebras}

 The aim of this section is to explain that certain Noetherianity
or coherence-type conditions on a conilpotent coalgebra $C$ imply
weakly finite Koszulity.
 In particular, all \emph{cocommutative} conilpotent coalgebras are
weakly finitely Koszul; this fact will be relevant for the discussion
in the next Section~\ref{cocommutative-secn}.

 Let $C$ be a coalgebra over a field~$k$.
 A left $C$\+comodule is said to be \emph{finitely cogenerated} if it
can be embedded as a subcomodule into a cofree $C$\+comodule $C\ot_k V$
with a finite-dimensional space of cogenerators~$V$.
 A coalgebra $C$ is said to be \emph{left co-Noetherian} if every
quotient comodule of a finitely cogenerated left $C$\+comodule is
finitely cogenerated~\cite{WW,GTNT,Pmc}.
 A coalgebra $C$ is said to be \emph{left Artinian} if it is
Artinian as an object of the category of left $C$\+comodules
$C\Comodl$, that is, any descending chain of left coideals in $C$
terminates~\cite{GTNT}, \cite[Section~2]{Pmc}.
 (Here by a \emph{left coideal} in $C$ one means a left subcomodule
in the left $C$\+comodule~$C$.)

\begin{lem} \label{co-Noetherian-lemma}
\textup{(a)} A coalgebra $C$ is left Artinian if and only if it is
left co-Noetherian \emph{and} the maximal cosemisimple subcoalgebra
of $C$ is finite-dimensional.
 In particular, a conilpotent coalgebra is left Artinian if and
only if it is left co-Noetherian. \par
\textup{(b)} A coalgebra $C$ is left Artinian if and only if its
dual algebra $C^*$ is right Noetherian.
\end{lem}

\begin{proof}
 See~\cite[Proposition~1.6 of the published version,
or Proposition~2.5 of the \texttt{arXiv} version]{GTNT}
or~\cite[Lemmas~2.6 and~2.10(b), Example~2.7, and the general
discussion in Section~2]{Pmc}.
\end{proof}

\begin{prop} \label{co-Noetherian-are-wf-koszul}
 Any left or right co-Noetherian conilpotent coalgebra is weakly
finitely Koszul.
\end{prop}

\begin{proof}
 The weak finite Koszulity property of a conilpotent coalgebra is
left-right symmetric by Proposition~\ref{ext-k-k-left-right-symmetric};
so it suffices to consider the case of a left co-Noetherian conilpotent
coalgebra~$C$.
 Then it follows easily from the co-Noetherianity that every finitely
cogenerated left $C$\+comodule $L$ has a coresolution $J^\bu$ by finitely
cogenerated cofree $C$\+comodules $J^n$, \,$n\ge0$.
 In particular, this applies to all finite-dimensional left
$C$\+comodules~$L$; so the left $C$\+comodule~$k$ has such
a coresolution~$J^\bu$.
 It remains to compute the spaces $\Ext^n_C(k,k)$ as
$H^n(\Hom_C(k,J^\bu))$ and notice that the vector space
$\Hom_C(k,J^n)={}_\gamma J^n$ is finite-dimensional for every $n\ge0$.
\end{proof}

 A left $C$\+comodule is said to be \emph{finitely copresented} if it
can be obtained as the kernel of a morphism of cofree $C$\+comodules
$C\ot_k V\rarrow C\ot_k U$ with finite-dimensional vector spaces of 
cogenerators $V$ and~$U$.
 A coalgebra $C$ is said to be \emph{left cocoherent} if every
finitely cogenerated quotient comodule of a finitely copresented left
$C$\+comodule is finitely copresented~\cite[Section~2]{Pmc}.

\begin{lem} \label{finitely-copresented-over-conilpotent}
 Let $C$ be a conilpotent coalgebra over~$k$.
 Then \par
\textup{(a)} a left $C$\+comodule $L$ is finitely cogenerated if and
only if the vector space\/ $\Hom_k(k,L)={}_\gamma L$
is finite-dimensional; \par
\textup{(b)} a finitely cogenerated left $C$\+comodule $L$ is
finitely copresented if and only if the vector space\/ $\Ext_C^1(k,L)$
is finite-dimensional; \par
\textup{(c)} in particular, the left $C$\+comodule~$k$ is finitely
copresented if and only if the coalgebra $C$ is finitely cogenerated.
\end{lem}

\begin{proof}
 Part~(a) follows from
Lemma~\ref{one-step-co-free-co-resolution}(a).
 To prove the ``if'' assertion in part~(b), one can use
Proposition~\ref{minimal-co-resolutions-prop}(a) together with
Corollary~\ref{inj-proj-co-free-cor}(a) to the effect
that for any left $C$\+comodule $L$ there exists a left exact
sequence of left $C$\+comodules $0\rarrow L\rarrow C\ot V_0\rarrow
C\ot_k V_1$ with $V_0\simeq\Hom_C(k,L)$ and $V_1\simeq\Ext_C^1(k,L)$.
 The ``only if'' assertion in part~(b) is deduced from the fact that
the quotient comodule of a finitely cogenerated comodule by a finitely
copresented subcomodule is finitely cogenerated;
see~\cite[Theorem~6]{WW} or~\cite[Lemma~2.8(a)]{Pmc}.
 In particular, the left $C$\+comodule~$k$ is finitely copresented
if and only if the vector space $\Hom_k(k,C/\gamma(k))=
{}_\gamma(C/{}_\gamma C)=\Ext^1_C(k,k)$ is finite-dimensional.
 This is part~(c).
\end{proof}

\begin{prop} \label{fin-cogen-cocoherent-are-wf-koszul}
 Any finitely cogenerated left or right cocoherent conilpotent
coalgebra is weakly finitely Koszul.
\end{prop}

\begin{proof}
 It suffices to consider the case when $C$ is left cocoherent.
 Then it follows from the cocoherence and the same fact that
the quotient comodule of a finitely cogenerated comodule by a finitely
copresented subcomodule is finitely cogenerated (mentioned in
the proof of Lemma~\ref{finitely-copresented-over-conilpotent}) that
any finitely copresented left $C$\+comodule $L$ has a coresolution
$J^\bu$ by finitely cogenerated cofree $C$\+comodules $J^n$, \,$n\ge0$.
 If $C$ is finitely cogenerated, then this applies to $L=k$
(by Lemma~\ref{finitely-copresented-over-conilpotent}(c)); so
the left $C$\+comodule $k$ has such a coresolution $J^\bu$.
 The argument finishes exactly in the same way as the proof of
Proposition~\ref{co-Noetherian-are-wf-koszul}.
\end{proof}

\begin{rem}
 A coalgebra version of the Morita equivalence theory was developed
by Takeuchi~\cite{Tak}.
 The notions of co-Noetherianity and co-coherence for coalgebras are
\emph{not} Morita--Takeuchi invariant (neither is conilpotence, of
course; while Artinianity of coalgebras \emph{is} invariant under
the Morita--Takeuchi equivalence).

 The co-Noetherianity and cocoherence properties have Morita--Takeuchi
invariant versions, which are called \emph{quasi-co-Noetherianity}
and \emph{quasi-cocoherence} in~\cite[Section~3]{Ppc},
\cite[Sections~5.1\+-5.4]{Psm}.
 (The former property was called ``strict quasi-finiteness''
in~\cite{GTNT}.)
 The key definition, going back to Takeuchi~\cite{Tak}, is that of
a ``quasi-finite comodule'', called \emph{quasi-finitely cogenerated}
in the terminology of~\cite{Ppc,Psm}.
 A left $C$\+comodule $L$ is called quasi-finitely cogenerated if
the vector space $\Hom_C(K,L)$ is finite-dimensional for any
finite-dimensional left $C$\+comodule~$K$.

 One can see from~\cite[Lemma~2.2(e)]{Pmc} that over a coalgebra $C$
with finite-dimen\-sional maximal cosemisimple subcoalgebra (in
particular, over a conilpotent coalgebra~$C$) the classes of
finitely cogenerated and quasi-finitely cogenerated comodules coincide.
 Consequently, a conilpotent coalgebra is quasi-co-Noetherian in
the sense of~\cite{Ppc,Psm} if and only if it is co-Noetherian, and
a conilpotent coalgebra is quasi-cocoherent in the the sense
of~\cite{Ppc,Psm} if and only if it is cocoherent.
 That is why we were not concerned with the quasi-co-Noetherianity
and quasi-cocoherence properties in this section, but only with
the co-Noetherianity and cocoherence properties.
\end{rem}

\Section{Cocommutative Conilpotent Coalgebras}
\label{cocommutative-secn}

 Let $V$ be a $k$\+vector space.
 We refer to~\cite[Sections~2.3 and~3.3]{Pksurv} for an introductory
discussion of the \emph{cofree conilpotent} (\emph{tensor}) coalgebra
$$
 \Ten(V)=\bigoplus\nolimits_{n=0}^\infty V^{\ot n}
$$
cospanned by a vector space~$V$.
 The \emph{cofree conilpotent cocommutative} (\emph{symmetric})
coalgebra $\Sym(V)$ is defined as the subcoalgebra
$$
 \Sym(V)=\bigoplus\nolimits_{n=0}^\infty\Sym^n(V)\subset\Ten(V),
$$
where $\Sym^n(V)\subset V^{\ot n}$ is the vector subspace of all
symmetric tensors in~$V^{\ot n}$.
 One can easily see that $\Sym(V)$ is the maximal cocommutative
subcoalgebra in $\Ten(V)$; in fact, for any cocommutative coalgebra $C$
over~$k$ and coalgebra homomorphism $f\:C\rarrow\Ten(V)$, the image
of~$f$ is contained in $\Sym(V)$.

 Let $C$ be a conilpotent coalgebra over~$k$.
 Assume that $C$ is finitely cogenerated; so the vector space
$V=\Ext^1_C(k,k)=\ker(C_+\to C_+\ot_k C_+)$ is finite-dimensional.
 Choose a $k$\+linear projection $g\:C_+\rarrow V$ onto the vector
subspace $V\subset C_+$.
 Then, according to~\cite[Lemma~5.2(a)]{Pqf}, the map~$g$ extends
uniquely to coalgebra homomorphism $f\:C\rarrow\Ten(V)$.
 Moreover, by~\cite[Lemma~5.2(b)]{Pqf}, the map~$f$ is injective.
 
 Assume additionally that $C$ is cocommutative.
 Then, following the discussion above, the image of the map~$f$ is
contained in the subcoalgebra $\Sym(V)\subset\Ten(V)$.
 Hence $C$ is a subcoalgebra in $\Sym(V)$.

 Choose a $k$\+vector space basis $x_1^*$,~\dots, $x_m^*$ in
the vector space~$V$.
 Then $x_1$,~\dots, $x_m$ is a basis in the dual vector space~$V^*$.
 The choice of such bases identifies the $k$\+vector space dual
$k$\+algebra $\Sym(V)^*$ to the symmetric coalgebra $\Sym(V)$ with
the algebra of formal power series in the variables~$x_1$,~\dots,~$x_m$,
$$
 \Sym(V)^*\simeq k[[x_1,\dotsc,x_m]].
$$
 Accordingly, the $k$\+algebra $C^*$ dual to $C$ is a quotient algebra
of the algebra of formal power series by an ideal
$J\subset k[[x_1,\dotsc,x_m]]$,
$$
 C^*\simeq k[[x_1,\dotsc,x_m]]/J.
$$

 Denote by $s_1$,~\dots, $s_m\in C^*$ the images of the elements
$x_1$,~\dots, $x_m$ under the surjective $k$\+algebra homomorphism
$k[[x_1,\dotsc,x_m]]\rarrow C^*$ dual to the injective coalgebra map
$C\rarrow\Sym(V)$.
 We have shown that $C^*$ is a complete Noetherian commutative local
ring with the maximal ideal $I=(s_1,\dotsc,s_m)\subset C^*$
generated by the elements $s_1$,~\dots, $s_m\in C^*$.
 Indeed, the ring of formal power series $k[[x_1,\dotsc,x_m]]$ is
a complete Noetherian commutative local ring; hence so is any (nonzero)
quotient ring of $k[[x_1,\dotsc,x_m]]$.

\begin{cor} \label{fin-cogen-cocomm-conilp-cor}
 All finitely cogenerated cocommutative conilpotent coalgebras are
co-Noetherian.
 Consequently, all such coalgebras are weakly finitely Koszul.
\end{cor}

\begin{proof}
 Let $C$ be a finitely cogenerated cocommutative conilpotent coalgebra
over~$k$.
 Then the algebra $C^*$ is Noetherian, as explained above.
 By Lemma~\ref{co-Noetherian-lemma}(a\+-b), it follows that
the coalgebra $C$ is co-Noetherian.
 Now Proposition~\ref{co-Noetherian-are-wf-koszul} tells us that $C$ is
weakly finitely Koszul.
\end{proof}

 Corollary~\ref{fin-cogen-cocomm-conilp-cor} says that the condition
of Theorem~\ref{bounded-half-unbounded-derived-main-theorem}(vii) is
satisfied for any finitely cogenerated cocommutative conilpotent
coalgebra~$C$.
 Consequently, the triangulated functors $\Upsilon^+\:\sD^+(C\Comodl)
\rarrow\sD^+(C^*\Modl)$ and $\Theta^-\:\sD^-(C\Contra)\rarrow
\sD^-(C^*\Modl)$ are fully faithful.
 In the rest of this section, our aim is to show, using the results of
the paper~\cite{Pmgm}, that the triangulated functors
$\Upsilon^\varnothing\:\sD(C\Comodl)\rarrow\sD(C^*\Modl)$ and
$\Theta^\varnothing\:\sD(C\Contra)\rarrow\sD(C^*\Modl)$ between
the unbounded derived categories are actually fully faithful in this
case, too.

 Let $R$ be a commutative ring and $I\subset R$ be an ideal.
 An $R$\+module $M$ is said to be \emph{$I$\+torsion} if for every
pair of elements $s\in I$ and $x\in M$ there exists an integer $n\ge1$
such that $s^nx=0$ in~$M$.
 Equivalently, $M$ is $I$\+torsion if and only if for every $s\in I$
one has $R[s^{-1}]\ot_RM=0$.
 Here $R[s^{-1}]$ is the ring obtained by inverting formally the element
$s\in R$ (or equivalently, by localizing $R$ at the multiplicative
subset $S=\{1,s,s^2,s^3,\dotsc\}\subset R$ spanned by~$s$).

 An $R$\+module $P$ is said to be
an \emph{$I$\+contramodule}~\cite{Pcta,Pmgm} if for every element
$s\in I$ one has
$$
 \Hom_R(R[s^{-1}],P)=0=\Ext^1_R(R[s^{-1}],P).
$$
 It is important to notice here that the projective dimension of
the $R$\+module $R[s^{-1}]$ can never exceed~$1$
\,\cite[proof of Lemma~2.1]{Pcta}.

 The full subcategory $R\Modl_{I\tors}$ of all $I$\+torsion
$R$\+modules is closed under submodules, quotients, extensions,
and infinite direct sums in $R\Modl$, as one can easily
see~\cite[Theorem~1.1(b)]{Pcta}.
 The full subcategory $R\Modl_{I\ctra}$ of all $I$\+contramodule
$R$\+modules is closed under kernels, cokernels, extensions,
and infinite products in $R\Modl$
\,\cite[Proposition~1.1]{GL}, \cite[Theorem~1.2(a)]{Pcta}.
 Consequently, both the categories $R\Modl_{I\tors}$ and
$R\Modl_{I\ctra}$ are abelian, and the inclusion functors
$R\Modl_{I\tors}\rarrow R\Modl$ and $R\Modl_{I\ctra}\rarrow R\Modl$
are exact.

\begin{thm} \label{essential-images-described}
 Let $C$ be a finitely cogenerated cocommutative conilpotent coalgebra
over a field~$k$, and let $I\subset C^*$ be the maximal ideal of
the complete Noetherian commutative local ring~$C^*$.
 Then \par
\textup{(a)} the essential image of the comodule inclusion functor\/
$\Upsilon\:C\Comodl\rarrow C^*\Modl$ coincides with the full subcategory
of $I$\+torsion $C^*$\+modules $C^*\Modl_{I\tors}\subset C^*\Modl$,
so the functor\/ $\Upsilon$ induces a category equivalence
$C\Comodl\simeq C^*\Modl_{I\tors}$; \par
\textup{(b)} the contramodule forgetful functor\/ $\Theta\:C\Contra
\rarrow C^*\Modl$ is fully faithful, and its essential image coincides
with the full subcategory of $I$\+contramodule $C^*$\+modules
$C^*\Modl_{I\ctra}\subset C^*\Modl$, so the functor\/ $\Theta$
induces a category equivalence $C\Contra\simeq C^*\Modl_{I\ctra}$.
\end{thm}

\begin{proof}
 The dual vector space $W^*$ to any (discrete infinite-dimensional)
vector space $W$ is naturally endowed with a pro-finite-dimensional
(otherwise known as linearly compact or pseudocompact) topology.
 The annihilators of finite-dimensional vector subspaces in $W$ are
precisely all the open vector subspaces in $W^*$, and they form
a base of neighborhoods of zero in~$W^*$.
 In the case of a coalgebra $C$, the vector space $C^*$ with its
pro-finite-dimensional topology is a topological algebra.

 A left $C^*$\+module $N$ is said to be \emph{discrete} (or ``rational''
in the terminology of~\cite{Swe}) if, for every element $x\in N$,
the annihilator of~$x$ is an open left ideal in~$C^*$.
 Equivalently, this means that the action map $C^*\times N\rarrow N$
is continuous in the pro-finite-dimensional topology on $C^*$ and
the discrete topology on~$N$.
 The essential image of the comodule inclusion functor $\Upsilon\:
C\Comodl\rarrow C^*\Modl$ consists precisely of all the discrete
$C^*$\+modules~\cite[Propositions~2.1.1\+-2.1.2]{Swe}.
 (Cf.\ the proof of
Theorem~\ref{bounded-half-unbounded-derived-main-theorem}
in Section~\ref{half-bounded-secn}.)
 These assertions hold for any coalgebra $C$ over~$k$.

 In the case of a finitely cogenerated cocommutative conilpotent
coalgebra $C$, the surjective map $k[[x_1,\dotsc,x_m]]\rarrow C^*$
is open and continuous, because it is obtained by applying the dual
vector space functor $W\longmapsto W^*=\Hom_k(W,k)$ to the inclusion
of discrete vector spaces (coalgebras) $C\rarrow\Sym(V)$.
 The resulting pro-finite-dimensional topology on the algebra of formal
power series $k[[x_1,\dotsc,x_m]]$ is the usual (adic) topology of
the formal power series.
 One easily concludes that a $C^*$\+module is discrete if and only if
it is discrete over $k[[x_1,\dotsc,x_m]]$.
 By the definition of the adic topology on a formal power series ring
$k[[x_1,\dotsc,x_m]]$, a module over this ring is discrete if and
only if it is a torsion module for the maximal ideal
$(x_1,\dotsc,x_m)\subset k[[x_1,\dotsc,x_m]]$.
 Hence a $C^*$\+module $N$ is discrete if and only if it is a torsion
module for the ideal $I=(s_1,\dotsc,s_m)\subset C^*$.
 One can also observe that the pro-finite-dimensional topology on
$C^*$ coincides with the $I$\+adic topology, since this holds for
the coalgebra $\Sym(V)$.
 This proves part~(a).

 To prove part~(b), one needs to use the concept of a \emph{contramodule
over a topological ring} as an intermediate step.
 Without going into the (somewhat involved) details of this definition,
which can be found in~\cite[Section~2.1]{Pweak},
\cite[Section~2.1]{Prev}, or~\cite[Sections~2.5\+-2.7]{Pproperf},
let us say that the category of left $C$\+contramodules is naturally
equivalent (in fact, isomorphic) to the category of left contramodules
over the topological ring~$C^*$.
 This equivalence agrees with the natural forgetful functors acting
from the categories of left $C$\+contramodules and left
$C^*$\+contramodules to the category of left $C^*$\+modules.
 We refer to~\cite[Section~1.10]{Pweak} or~\cite[Section~2.3]{Prev}
for the details of the proof of this assertion.

 The contramodule forgetful functor $C\Contra\rarrow C^*\Modl$ is
fully faithful by Theorem~\ref{fully-faithful-contramodule-forgetful}.
 In the special case of \emph{cocommutative} coalgebras $C$, one can
also obtain this result as a particular case of the following
theorems, which provide more information.
 By~\cite[Theorem~B.1.1]{Pweak} or~\cite[Theorem~2.2]{Prev},
the forgetful functor from the category of $C^*$\+contramodules to
the category of $C^*$\+modules is fully faithful, and its essential
image is precisely the full subcategory of all $I$\+contramodule
$C^*$\+modules.
 These results are actually applicable to any commutative Noetherian
ring $R$ with a fixed ideal $I$ and the $I$\+adic completion of $R$
viewed as a topological ring with the $I$\+adic topology.
 (Morever, the Noetherianity condition can be weakened and replaced
with a certain piece of the weak proregularity condition;
see~\cite[Proposition~1.5, Corollary~3.7, and Remark~3.8]{Pdc}.)
 The combination of the references in this paragraph and in
the previous one establishes part~(b).
\end{proof}

\begin{thm} \label{commutative-algebra-unbounded-derived-ff}
 Let $R$ be a commutative Noetherian ring and $I\subset R$ be an ideal.
 Then \par
\textup{(a)} the triangulated functor between the unbounded derived
categories
$$
 \sD(R\Modl_{I\tors})\lrarrow\sD(R\Modl)
$$
induced by the inclusion of abelian categories $R\Modl_{I\tors}\rarrow
R\Modl$ is fully faithful; \par
\textup{(b)} the triangulated functor between the unbounded derived
categories
$$
 \sD(R\Modl_{I\ctra})\lrarrow\sD(R\Modl)
$$
induced by the inclusion of abelian categories $R\Modl_{I\ctra}\rarrow
R\Modl$ is fully faithful.
\end{thm}

\begin{proof}
 These assertions actually hold for any \emph{weakly proregular}
finitely generated ideal $I$ in a (not necessarily Noetherian)
commutative ring~$R$; while in a Noetherian commutative ring, any ideal
is weakly proregular.
 See~\cite[Theorems~1.3 and~2.9]{Pmgm} for the details.
 The proofs are based on the observations that (the derived functors of)
the right adjoint functor $\Gamma\:R\Modl\rarrow R\Modl_{I\tors}$
and the left adjoint functor $\Delta\:R\Modl\rarrow R\Modl_{I\ctra}$ to
the inclusions of abelian categories in question have finite
homological dimensions, which go back to~\cite[Corollaries~4.28
and~5.27]{PSY} and were also mentioned in~\cite[Lemmas~1.2(b)
and~2.7(b)]{Pmgm}.
 See~\cite[Theorem~6.4]{PMat} (cf.~\cite[Proposition~6.5]{Pper}) for
an abstract formulation.
\end{proof}

\begin{cor}
 Let $C$ be a finitely cogenerated cocommutative conilpotent coalgebra
over a field~$k$.
 Then \par
\textup{(a)} the triangulated functor between the unbounded derived
categories
$$
 \Upsilon^\varnothing\:\sD(C\Comodl)\lrarrow\sD(C^*\Modl)
$$
induced by the comodule inclusion functor\/ $\Upsilon\:C\Comodl\rarrow
C^*\Modl$ is fully faithful; \par
\textup{(b)} the triangulated functor between the unbounded derived
categories
$$
 \Theta^\varnothing\:\sD(C\Contra)\lrarrow\sD(C^*\Modl)
$$
induced by the contramodule forgetful functor $C\Contra\rarrow
C^*\Modl$ is fully faithful.
\end{cor}

\begin{proof}
 For part~(a), compare the results of
Theorems~\ref{essential-images-described}(a) and
\ref{commutative-algebra-unbounded-derived-ff}(a).
 For part~(b), similarly compare
Theorems~\ref{essential-images-described}(b) and
\ref{commutative-algebra-unbounded-derived-ff}(b).
\end{proof}

\appendix

\bigskip
\section*{Appendix.  Standard Category-Theoretic Observations}
\medskip

\setcounter{section}{1}
\setcounter{thm}{0}

 In this appendix we collect several elementary category-theoretic
homological algebra lemmas, which are used in
Sections~\ref{wf-koszulity-implies-comodule-ext-isom-secn}\+-%
\ref{contramodule-ext-isom-implies-wf-koszulity-secn}.
 The following lemma tells how far one can go computing the Ext groups
with a resolution which is only \emph{initially} projective.

\begin{lem} \label{initially-projective-resolution-computes-ext}
 Let\/ $\sA$ be an abelian category and\/
$0\larrow X\larrow P_0\larrow P_1\larrow P_2\larrow\dotsb$
be an exact sequence in\/~$\sA$.
 Then for every object $Y\in\sA$ there are natural maps of abelian
groups
\begin{equation} \label{resolution-to-ext-comparison}
 H^n(\Hom_\sA(P_\bu,Y))\lrarrow\Ext^n_\sA(X,Y)
\end{equation}
defined for all $n\ge0$.
 The map~\eqref{resolution-to-ext-comparison} is an isomorphism
whenever the objects $P_i$ are projective in\/ $\sA$ for all\/
$0\le i\le n-1$.
\end{lem}

\begin{proof}
 The point of the lemma is that one does \emph{not} need the objects
$P_n$ or $P_{n+1}$ (but only the objects $P_0$,~\dots, $P_{n-1}$) to be
projective in order to compute $\Ext^n_\sA(X,Y)$ using a resolution
$P_\bu$ of an object $X\in\sA$.
 The proof is standard.
 Denote by $X_i$ the image of the differential $P_i\rarrow P_{i-1}$
(so $X=X_0$).
 Then the connecting homomorphisms in the long exact sequences of
groups $\Ext^*_\sA({-},Y)$ related to the short exact sequences
$0\rarrow X_{i+1}\rarrow P_i\rarrow X_i\rarrow0$ in $\sA$ provide
natural maps of abelian groups
{\setlength{\multlinegap}{0em}
\begin{multline*}
 H^n(\Hom_\sA(P_\bu,Y))=
 \coker(\Hom_\sA(P_{n-1},Y)\to\Hom_\sA(X_n,Y))
 \lrarrow\Ext^1_\sA(X_{n-1},Y) \\ \lrarrow
 \Ext^2_\sA(X_{n-2},Y) \lrarrow\dotsb\rarrow
 \Ext^{n-1}_\sA(X_1,Y)\lrarrow\Ext^n_\sA(X_0,Y),
\end{multline*}
whose composition} is the desired
map~\eqref{resolution-to-ext-comparison}.
 Furthermore, the map $H^n(\Hom_\sA(P_\bu,Y))\allowbreak\rarrow
\Ext^1_\sA(X_{n-1},Y)$ is an isomorphism whenever
$\Ext^1_\sA(P_{n-1},Y)=0$; the map
$\Ext^1_\sA(X_{n-1},Y)\rarrow\Ext^2_\sA(X_{n-2},Y)$ is
an isomorphism whenever $\Ext^1_\sA(P_{n-2},Y)=0=\Ext^2_\sA(P_{n-2},Y)$,
etc.; and the map $\Ext^{n-1}_\sA(X_1,Y)\rarrow\Ext^n_\sA(X_0,Y)$
is an isomorphism whenever $\Ext^{n-1}_\sA(P_0,Y)=0=\Ext^n_\sA(P_0,Y)$.
\end{proof}

 The next lemma is a Tor version of
Lemma~\ref{initially-projective-resolution-computes-ext}.

\begin{lem} \label{initially-projective-resolution-computes-tor}
 Let $R$ be an associative ring and\/ $0\larrow X\larrow F_0
\larrow F_1\larrow F_2\larrow\dotsb$ be an exact sequence of right
$R$\+modules.
 Then for every left $R$\+module $Y$ there are natural maps of
abelian groups
\begin{equation} \label{tor-to-resolution-comparison}
  \Tor^R_n(X,Y)\lrarrow H_n(F_\bu\ot_RY)
\end{equation}
defined for all $n\ge0$.
 The map~\eqref{tor-to-resolution-comparison} is an isomorphism
whenever the $R$\+modules $F_i$ are flat for all\/ $0\le i\le n-1$. \qed
\end{lem}

 The following lemma is very easy.

\begin{lem} \label{induced-on-Ext-1-lemma}
 Let\/ $\Phi\:\sB\rarrow\sA$ be a fully faithful exact functor of
abelian categories, and let $X$, $Y\in\sB$ be two objects.
 Then the map
$$
 \Ext^1_\sB(X,Y)\lrarrow\Ext^1_\sA(\Phi(X),\Phi(Y))
$$
induced by the functor\/~$\Phi$ is injective.
 This map is surjective if and only if, for any short exact sequence\/
$0\rarrow\Phi(Y)\rarrow A\rarrow\Phi(X)\rarrow0$ in\/ $\sA$,
the object $A$ belongs to the essential image of\/~$\Phi$.
\end{lem}

\begin{proof}
 The proof is left to the reader.
\end{proof}

 The next lemma is essentially well-known, but we give it a more
precise formulation than it usually receives.

\begin{lem} \label{next-degree-ext-injectivity-lemma}
 Let\/ $\sA$ and\/ $\sB$ be abelian categories, and let\/
$\Phi\:\sB\rarrow\sA$ be an exact functor.
 Let $n\ge1$ be an integer and $Y\in\sB$ be a fixed object. \par
\textup{(a)} Assume that the map of groups\/ $\Ext^{n-1}$
$$
 \Ext^{n-1}_\sB(X,Y)\lrarrow\Ext^{n-1}_\sA(\Phi(X),\Phi(Y))
$$
induced by the functor\/ $\Phi$ is an isomorphism for all objects
$X\in\sB$.
 Then the map of groups\/ $\Ext^n$
$$
 \Ext^n_\sB(X,Y)\lrarrow\Ext^n_\sA(\Phi(X),\Phi(Y))
$$
induced by the functor\/ $\Phi$ is injective for all objects $X\in\sB$. 
\par
\textup{(b)} More generally, let\/ $0\rarrow X\rarrow P\rarrow T
\rarrow0$ be a short exact sequence in\/ $\sB$.
 Assume that the map\/ $\Ext^{n-1}_\sB(P,Y)\rarrow
\Ext^{n-1}_\sA(\Phi(P),\Phi(Y))$ is surjective, while the map
$\Ext^{n-1}_\sB(X,Y)\rarrow\Ext^{n-1}_\sA(\Phi(X),\Phi(Y))$
is injective.
 Then the intersection of the kernel of the map\/
$\Ext^n_\sB(T,Y)\rarrow\Ext^n_\sB(P,Y)$ induced by the epimorphism
$P\rarrow T$ with the kernel of the map\/ $\Ext^n_\sB(T,Y)\rarrow
\Ext^n_\sA(\Phi(T),\Phi(Y))$ induced by the functor\/ $\Phi$ is
the zero subgroup in\/ $\Ext^n_\sB(T,Y)$.
\end{lem}

\begin{proof}
 This observation goes back, at least,
to~\cite[Remarque~3.1.17(i)]{BBD}.
 Part~(b) is provable by a straightforward diagram chase of
the commutative diagram of a morphism of long exact sequences
$$
 \xymatrixcolsep{1.25em}
 \xymatrix{
  \Ext^{n-1}_\sB(P,Y) \ar[r] \ar@{->>}[d]
  & \Ext^{n-1}_\sB(X,Y) \ar[r] \ar@{>->}[d]
  & \Ext^n_\sB(T,Y) \ar[r] \ar[d] & \Ext^n_\sB(P,Y) \\
  \Ext^{n-1}_\sA(\Phi(P),\Phi(Y)) \ar[r]
  & \Ext^{n-1}_\sA(\Phi(X),\Phi(Y)) \ar[r]
  & \Ext^n_\sA(\Phi(T),\Phi(Y))
 }
$$
induced by the functor~$\Phi$.

 To deduce~(a) from~(b), let $T\in\sB$ be an object and
$\beta\in\Ext^n_\sB(T,Y)$ be an extension class annihilated by the map
$\Ext^n_\sB(T,Y)\rarrow\Ext^n_\sA(\Phi(T),\Phi(Y))$.
 Choose an epimorphism $P\rarrow T$ in $\sB$ such that~$\beta$
is annihilated by the induced map $\Ext^n_\sB(T,Y)\rarrow
\Ext^n_\sB(P,Y)$.
 Then part~(b) implies that $\beta=0$.
\end{proof}

 The final series of lemmas concerns abelian categories with enough
injective or projective objects.

\begin{lem} \label{ext-preservation-enough-for-injectives}
 Let\/ $\sA$ and\/ $\sB$ be abelian categories, and let\/
$\Phi\:\sB\rarrow\sA$ be a fully faithful exact functor.
 Assume that there are enough injective objects in the abelian
category\/ $\sB$, and denote by $\sB_\inj\subset\sB$ the class of
all such injective objects.
 Let $n\ge1$ be an integer.
 Then the following two conditions are equivalent:
\begin{enumerate}
\item the map of Ext groups
$$
 \Ext^i_\sB(X,Y)\lrarrow\Ext^i_\sA(\Phi(X),\Phi(Y))
$$
induced by the functor\/ $\Phi$ is an isomorphism for all objects
$X$, $Y\in\sB$ and integers\/ $0\le i\le n$;
\item $\Ext^i_\sA(\Phi(X),\Phi(J))=0$ for all objects $X\in\sB$
and $J\in\sB_\inj$, and all\/ $1\le i\le n$.
\end{enumerate}
\end{lem}

\begin{proof}
 (1)\,$\Longrightarrow$\,(2)  Condition~(2) is clearly the particular
case of~(1) for $Y=J\in\sB_\inj$.

 (2)\,$\Longrightarrow$\,(1)  For $i=0$, the map of Ext groups
in question is an isomorphism by the assumption that the functor
$\Phi$ is fully faithful.
 For $1\le i\le n$, we proceed by increasing induction on~$i$.

 Given an object $Y\in\sB$, choose a short exact sequence $0\rarrow
Y\rarrow J\rarrow Y'\rarrow0$ in $\sB$ with $J\in\sB_\inj$.
 Then the exact functor $\Phi$ induces a commutative diagram of
a morphism of long exact sequences
$$
 \xymatrix{
  \Ext^{i-1}_\sB(X,J) \ar[r] \ar@{=}[d]
  & \Ext^{i-1}_\sB(X,Y') \ar[r] \ar@{=}[d]
  & \Ext^i_\sB(X,Y) \ar[r] \ar[d] & 0 \\
  \Ext^{i-1}_\sA(\Phi(X),\Phi(J)) \ar[r]
  & \Ext^{i-1}_\sA(\Phi(X),\Phi(Y')) \ar[r]
  & \Ext^i_\sA(\Phi(X),\Phi(Y)) \ar[r] & 0
 }
$$
 Here $\Ext^i_\sB(X,J)=0$ since $i\ge1$ and $J\in\sB_\inj$, while
$\Ext^i_\sA(\Phi(X),\Phi(J))=0$ by~(2).
 The leftmost and the middle vertical morphisms are isomorphisms
by the induction assumption.
 It follows that the rightmost vertical morphism is an isomorphism, too
(as desired).
\end{proof}

 The following lemma is the dual version of
Lemma~\ref{ext-preservation-enough-for-injectives}, but with
a more precise claim.

\begin{lem} \label{ext-preservation-enough-for-projectives}
 Let\/ $\sA$ and\/ $\sB$ be abelian categories, and let\/
$\Phi\:\sB\rarrow\sA$ be a fully faithful exact functor.
 Assume that there are enough projective objects in the abelian
category\/ $\sB$, and denote by $\sB_\proj\subset\sB$ the class of
all such projective objects.
 Let $n\ge1$ be an integer.
 Then, for any fixed object $Y\in\sB$, the following two conditions are
equivalent:
\begin{enumerate}
\item the map of Ext groups
$$
 \Ext^i_\sB(X,Y)\lrarrow\Ext^i_\sA(\Phi(X),\Phi(Y))
$$
induced by the functor\/ $\Phi$ is an isomorphism for all objects
$X\in\sB$ and integers\/ $0\le i\le n$;
\item $\Ext^i_\sA(\Phi(P),\Phi(Y))=0$ for all objects $P\in\sB_\proj$
and all\/ $1\le i\le n$.
\end{enumerate}
\end{lem}

\begin{proof}
 Observe that the proof of
Lemma~\ref{ext-preservation-enough-for-injectives} works for
a fixed object $X$, and then dualize.
\end{proof}

 Our last lemma is the $n=0$ counterpart of
Lemma~\ref{ext-preservation-enough-for-projectives}.

\begin{lem} \label{full-and-faithfulness-enough-for-projectives}
 Let\/ $\sA$ and\/ $\sB$ be abelian categories, and let\/
$\Phi\:\sB\rarrow\sA$ be an exact functor.
 Assume that there are enough projective objects in the abelian
category\/~$\sB$, and denote by\/ $\sB_\proj\subset\sB$ the class
of all such projective objects.
 Then, for any fixed object $Y\in\sB$, the following two conditions are
equivalent:
\begin{enumerate}
\item the map of Hom groups
$$
 \Hom_\sB(X,Y)\lrarrow\Hom_\sA(\Phi(X),\Phi(Y))
$$
induced by the functor\/ $\Phi$ is an isomorphism for all objects
$X\in\sB$;
\item the map of Hom groups
$$
 \Hom_\sB(P,Y)\lrarrow\Hom_\sA(\Phi(P),\Phi(Y))
$$
induced by the functor\/ $\Phi$ is an isomorphism for all objects
$P\in\sB_\proj$.
\end{enumerate}
\end{lem}

\begin{proof}
 Represent the object $X$ as the cokernel of a morphism of projective
objects in the category~$\sB$.
\end{proof}

\end{document}